\documentclass[11pt,reqno]{amsart}
\usepackage{amsmath,amssymb,amsfonts}
\usepackage{amsthm,bm}
\usepackage{mathrsfs}
\usepackage{algorithm}
\usepackage{algpseudocode}
\usepackage{enumitem}
\usepackage{color}
\usepackage{comment}
\usepackage{graphicx}

\newtheorem{lemma}{Lemma}[section]
\newtheorem{remark}{Remark}[section]
\newtheorem{definition}{Definition}[section]

\newtheorem{proposition}{Proposition}[section]
\newtheorem{theorem}{Theorem}[section]

\newcommand{\ii}{\mathrm{i}}
\newcommand{\dd}{\,\mathrm{d}}
\newcommand{\Op}{\operatorname{Op}}
\newcommand{\supp}{\operatorname{supp}}
\newcommand{\diag}{\operatorname{diag}}

\numberwithin{equation}{section}

\begin{document}
\title[Time-modulated refocusing in inhomogeneous medium]{Time-modulated refocusing in inhomogeneous medium}

\author[Hongyu Liu]{Hongyu Liu$^\P$ }

\address{Department of Mathematics, City University of Hong Kong, Hong Kong SAR, China}

\email{hongyu.liuip@gmail.com}

\author[Wei Wu]{Wei Wu$^\dag$} \address{School of Mathematics, Jilin University, 2699 Qianjin Street, 130012 Changchun, China}

\email{wei\_wu@jlu.edu.cn}

\date{\today} \subjclass{} \keywords{}

\begin{abstract}

In an inhomogeneous dissipative medium, different rays experience different round-trip attenuation, so an instantaneous time mirror may retrace rays without producing a balanced point focus. We solve this problem for a microlocally selected back-propagating transverse-magnetic component generated by a short pulsating modulation of a thin shell surrounding an inhomogeneous Debye medium with Ohmic conductivity. An exact decomposition of the Maxwell--Debye field isolates the pulse-generated return, and a semiclassical mode reduction yields a scalar returned operator whose principal symbol separates the temporal pulse response, shell modulation, and accumulated material loss. Our analysis shows that geometric distortion and amplitude attenuation are both involved in wave propagation, but the geometric distortion cancels between the outgoing and returning branches, whereas physical attenuation remains. A normalized directional energy, computed from pre-modulation measurements using the known lossless backward propagator, contains the same leading attenuation factor as the returned symbol. Its reciprocal defines a smooth, nonnegative, physically admissible modulation that equalizes the selected source-side gain up to a relative $O(h)$ error without requiring a complete attenuation model. With such a modulation, the field generated by a point source has a focal-scale leading profile independent of directional attenuation, with a unique maximum at the source and coherent return time. In every fixed focal window, the maximum is displaced by at most $O(h^2)$ in space and time, and the spatial half-amplitude area is $O(h^2)$.

\end{abstract}


\maketitle

\section{Introduction}

Time reversal is a wave-control principle in which a field emitted by a localized source is recorded, reversed in time, and re-emitted so that the resulting wave propagates back toward the source. The modern experimental framework was developed in ultrasonics and acoustics, where time-reversal mirrors were shown to produce robust spatiotemporal compression even in complicated propagation environments \cite{Fink1992,Fink1997,FinkEtAl2000,FinkPrada2001}. A particularly important observation is that heterogeneity and multiple scattering need not be detrimental, and the additional propagation paths can enlarge the effective aperture, improve robustness, and enhance the resolution of the refocused field \cite{DerodeRouxFink1995,BlomgrenPapanicolaouZhao2002,BalRyzhik2003}. Mathematical analyses have related the quality of the focus to propagation, observation, and decay properties of the underlying wave equation \cite{BardosFink2002}. These features make time reversal relevant whenever waves must be concentrated through an imperfectly known or strongly heterogeneous medium, including imaging, source localization, communications, nondestructive testing, and targeted energy delivery.

The same principle extends to electromagnetic waves. Electromagnetic time-reversal mirrors have been realized experimentally, and their propagation and focusing mechanisms have been developed theoretically \cite{LeroseyEtAl2004,deRosnyLeroseyFink2010}. A conventional time-reversal mirror, however, requires a recording stage and an array capable of retransmitting the time-reversed signals. Time-varying media provide a different mechanism. A rapid temporal change of constitutive parameters produces temporal reflection and frequency conversion, a phenomenon already present in early studies of velocity-modulated electromagnetic media and subsequently developed in the theory of temporal boundaries and spacetime metamaterials \cite{Morgenthaler1958,XiaoMaywarAgrawal2014, PlansinisDonaldsonAgrawal2015,CalozDeckLeger2020,GaliffiEtAl2022}. The instantaneous time mirror introduced experimentally in \cite{BacotLabousseEddiFinkFort2016} replaces a boundary array by a short, spatially distributed perturbation of the medium, and the perturbation generates a backward-propagating component throughout the modulated region. A mathematical framework for time reversal produced by time-dependent perturbations was established in \cite{BalFinkPinaud2019}.

Dissipation creates a basic obstruction to this mechanism. In a conservative reciprocal medium, reversal of the propagating phase can make the return paths coherent. In an attenuating medium, the returning wave is attenuated once more rather than amplified by the amount lost during the outgoing propagation. When the medium is inhomogeneous, different rays experience different travel histories and therefore acquire different amplitudes, hence geometric return of the rays alone does not imply a balanced point focus. Modified adjoint fields and preprocessing have been used to compensate attenuation in acoustic time reversal \cite{AmmariBretinGarnierWahab2011}, and spatially tailored
instantaneous time mirrors have been proposed for electromagnetic focusing in absorbing media \cite{WuNobreFortRileyCosten2022}. The issue addressed here is to derive, within a dispersive vector model, a modulation that is both physically admissible and computable from measurements available before the material pulse, and then to prove that the resulting returned field actually develops a point focus with quantified focal-scale behavior.

We consider a transverse-magnetic electromagnetic pulse generated by a point source at $\bm z_\star$ in a two-dimensional section of an inhomogeneous Debye medium with Ohmic conductivity. Throughout the semiclassical discussion, we consider the propagation of wave with high frequency $\omega\sim h^{-1}$ when $0<h\ll 1$. The medium is surrounded by a thin material shell, and at a mirror time $T_h$ the constitutive parameters in that shell are perturbed for a duration $\tau_h=O(h)$ by a spatial profile $m_h$. The short perturbation produces a back-propagating negative-TM component. The central question is the following: can $m_h$ be chosen causally and nonnegatively so that the selected returned directions have the same leading gain and consequently form a coherent point focus at the source? This question involves more than ray retracing. It requires separation of the positive and negative propagating modes, identification of the complete round-trip principal amplitude, recovery of its directional loss from accessible data, and control of the returned field on the wavelength scale.

By carefully choosing the modulation function $m_h$, we can guarantee a measurement-calibrated cancellation of directional attenuation at the level of the returned-wave symbol. For a selected source phase point $\bm\rho$, the semiclassical reduction gives a principal returned symbol of the form
\begin{equation*}
 j_m^+(\bm\rho) = \overline{p_\eta\bigl(\omega(\bm\rho)\bigr)}m\bigl(\bm X_T(\bm\rho)\bigr)
 \exp\left(-2\int_0^T\gamma\bigl(\bm\Phi^s(\bm\rho)\bigr)\dd s\right).
\end{equation*}
Here $p_\eta$ is the known temporal response of the material pulse, $m$ denotes the modulation function on the shell, $\bm X_T$ is the outgoing arrival map, and $\gamma$ is the physical attenuation along the ray. The geometric transport contributions on the positive and negative branches cancel, leaving the round-trip material loss displayed above. Rather than inserting this pathwise loss from a complete model of the attenuating medium, we construct a normalized directional energy $D_h(\theta)$ from the measured field by applying the known lossless backward propagator. Its leading term contains the same directional attenuation factor as the returned symbol. Defining the modulation on the calibrated curve by
\begin{equation*}
 m_h\bigl(\bm x_h^c(\theta)\bigr)=\frac{\alpha_h}{D_h(\theta)}
\end{equation*}
therefore yields, uniformly on the retained chart,
\begin{equation*}
 m_h\bigl(\bm x_h^c(\theta)\bigr)\exp\left[-2\int_0^{T_h}\gamma\bigl(\bm\Phi^s(\bm\Psi_{T_h}(\varpi,\theta))\bigr)\dd s\right]=\alpha_h c_h^{\rm av}+O(h\alpha_h).
\end{equation*}
Thus, after accounting for the known pulse response, the selected source-side directions have a common returned gain up to relative order $h$. This converts a direction-dependent lossy return into a balanced microlocal aperture by a causal, nonnegative spatial modulation.

The main contributions are as follows.
\begin{enumerate}[label=(\roman*),leftmargin=2.5em]
 \item
 We derive an exact decomposition of the Maxwell--Debye field into the unmodulated field, the component generated by the short material pulse, and a nonlinear material remainder. This identifies the returned operator that must be analyzed and separates the time-reversal mechanism from the remaining response of the modulated medium.

 \item
 We construct semiclassical injection and extraction operators for the positive and negative propagating TM branches and compress the vector returned operator to a scalar microlocal operator. Its principal symbol separates the temporal pulse response, the spatial modulation evaluated at the outgoing arrival point, and the accumulated round-trip attenuation. The opposite geometric transport terms cancel between the two branches, whereas the physical loss remains.

 \item
 At a specific time, we carry out observation, which is represented as $\mathcal O_{\mathcal A}^+$, on a given domain (which is assumed to be an annulus surrounding the source in the paper). Based on the observed data, we introduce the normalized directional measurement $D_h$, prove that it contains the same leading directional loss as the returned symbol, and use its reciprocal to define $m_h$. The resulting modulation is causal, smooth on the selected calibration curve, extendible as a compactly supported nonnegative function on the shell, and compatible with the positivity of the modulated material coefficients. The corresponding selected returned symbol is direction-independent up to a relative $O(h)$ error.

 \item
 For the point source, we obtain a focal-scale expansion of the selected negative-TM electric component in the variables
 $$
  \varsigma=\frac{t-t_{{\rm focus},h}}{h},
  \qquad
  \bm y=\frac{\bm x-\bm z_\star}{h}.
 $$
 The leading profile is independent of directional attenuation and has a unique global maximum at $(\varsigma,\bm y)=(0,0)$. In every fixed focal window, any maximizer of the selected field is displaced from the coherent return point by at most $O(h^2)$ in physical time and space, and the spatial half-amplitude area is of order $h^2$.
\end{enumerate}

The paper is organized as follows. Section~\ref{sec:prb-formulation} formulates the time-dependent Maxwell--Debye system, establishes the background evolution, and derives the returned component generated by the short material pulse. Section~\ref{sec:semiclassical} constructs the positive and negative TM scalar dynamics and computes the principal symbol of the returned operator. Section~\ref{sec:necessity-tailored-annulus} analyzes the point-source spectrum and the visible source-side chart, develops the causal directional measurement, and constructs the tailored modulation. Section~\ref{sec:causal-refocusing} transfers the cancellation back to the source and proves the focal-scale expansion, uniqueness of the leading focus, localization of the perturbed maximum, and the half-amplitude area estimates. Appendix~\ref{sec:semi-convention} records the semiclassical notation and operator conventions used in this manuscript.

\section{Problem Formulation}\label{sec:prb-formulation}

\subsection{The Maxwell system}

In this paper, we consider the propagation of electromagnetic wave in inhomogeneous medium. The wave is excited by an external source positioned at $\bm{z}_\star$ inside inhomogeneous medium. A thin shell surrounds the source from a distance. Let $\bm{E}=\bm{E}(t, \bm{x}), \bm{H}=\bm{H}(t, \bm{x})$ denote the electric and magnetic field, correspondingly. Due to the inhomogeneity, eddy current and time-dependent polarization is excited in the system, so the propagation of electromagnetic wave in the system is described by Maxwell equations
\begin{equation}\label{eq:maxwellmod}
\left\{
\begin{aligned}
\nabla\times\bm E &= -\mu\partial_t\bm H, \\
\nabla\times\bm H &= \partial_t\bm D+\bm J + \bm{J}_h^{\rm ext},
\end{aligned}
\right.
\end{equation}
together with time-domain constitutive equations
\begin{equation}\label{eq:modconstitutive}
\left\{
\begin{aligned}
&\bm D=\epsilon_0\epsilon_\infty(1+V_m)\bm E+\bm P, \\
&\tau_D\partial_t\bm P+\bm P = \epsilon_0\Delta\epsilon(1+V_m)\bm E, \\
&\bm J=\sigma(1+V_m)\bm E.
\end{aligned}
\right.
\end{equation}

The permeability is a fixed constant $\mu>0$. The material coefficients $\epsilon_\infty, \Delta\epsilon, \tau_D$ and $\sigma$ are real-valued smooth functions of $\bm{x}$ and satisfy
\begin{equation}\label{eq:modH}
\epsilon_\infty\geq c_->0,
\qquad
\Delta\epsilon\geq\delta_->0,
\qquad
0<\tau_-\le\tau_D\le\tau_+,
\qquad
0\le\sigma\le\sigma_+.
\end{equation}
Here $c_-, \delta_-, \tau_-, \tau_+$ and $\sigma_+$ are constants. 

The modulation function $V_m(t,\bm{x})$ is defined as $V_m(t,\bm{x}):=\eta_{\tau_h}(t-T)m(\bm{x})$. The modulation is supported only on the thin shell and is pulsating with very short time duration. After introducing semiclassical parameter $0<h\ll 1$ in section \ref{sec:semiclassical}, we set pulse duration $\tau_h:=\nu h$, where $\nu>0$ fixed. 

Let $\eta_{\tau_h}(t):=\tau_h^{-1}\eta(t/\tau_h)$, $\eta\in C_c^\infty((-1/2,1/2))$ depicts the pulsating modulation in time dimension with the pulse width $\tau_h$. We call $T$ the \emph{mirror time}. $T$ is the time of pulsating modulation center, and $m$ denotes the spatial modulation profile. $m$ and $T$ will be determined in Section \ref{sec:necessity-tailored-annulus}. 

To ensure the well-posedness of the system, we further require that for all $t$ and $\bm{x}$, 
\begin{equation}\label{eq:material-positivity}
1+V_m(t,\bm{x})\geq\kappa_0>0.
\end{equation}  

\begin{remark}
When $V_m$ is time-independent, system \eqref{eq:maxwellmod}--\eqref{eq:modconstitutive} is equivalent to a Debye dissipative medium with Ohmic conductivity. 
\end{remark}

Throughout the whole paper, we focus on transverse magnetic(TM) fields
$$
\bm E=(0,0,E_z)^\top,
\qquad
\bm H=(H_x,H_y,0)^\top,
\qquad
\bm P=(0,0,P)^\top.
$$
Set
\begin{align*}
&\bm{W}:=(H_x,H_y,E_z,P)^\top, &&
\epsilon_\Delta(t,\bm{x}):=\epsilon_0\Delta\epsilon(1+V_m(t,\bm{x})), \\
&\bm{e}_3:=(0,0,1,0)^\top, &&
\bm{e}_4:=(0,0,0,1)^\top.
\end{align*}
When the source term $\bm{J}_h^{\rm ext}$ is absent, \eqref{eq:maxwellmod}--\eqref{eq:modconstitutive} are reformulated as
\begin{equation}\label{eq:system}
\left\{
\begin{aligned}
&\mu\partial_tH_x+\partial_{x_2}E_z=0,\\
&\mu\partial_tH_y-\partial_{x_1}E_z=0,\\
&\epsilon_0\epsilon_\infty(1+V_m)\partial_tE_z
 +\partial_{x_2}H_x-\partial_{x_1}H_y+\left[\sigma(1+V_m)+\epsilon_0\epsilon_\infty\partial_tV_m
 +\frac{\epsilon_\Delta}{\tau_D}\right]E_z-\frac1{\tau_D}P=0,\\
&\partial_tP=\frac1{\tau_D}(\epsilon_\Delta E_z-P),
\end{aligned}
\right.
\end{equation}
and equivalently,
\begin{equation}\label{eq:modmatrix}
A_V(t)\partial_t\bm{W}+\Lambda\bm{W}+B_V(t)\bm{W}=0,
\end{equation}
where
\begin{equation}\label{eq:modAB}
\begin{aligned}
\Lambda&=
\begin{pmatrix}
0 & 0 & \partial_{x_2} & 0\\
0 & 0 & -\partial_{x_1} & 0\\
\partial_{x_2} & -\partial_{x_1} & 0 & 0\\
0 & 0 & 0 & 0
\end{pmatrix}, \qquad
A_V(t)=\diag\left(
\mu,\mu,\epsilon_0\epsilon_\infty(1+V_m),\frac{1}{\epsilon_\Delta}
\right),\\
B_V(t) & =
\begin{pmatrix}
0 & 0 & 0 & 0\\
0 & 0 & 0 & 0\\
0 & 0 & \sigma(1+V_m)+\epsilon_0\epsilon_\infty\partial_tV_m+
\dfrac{\epsilon_\Delta}{\tau_D} & -\dfrac1{\tau_D}\\
0 & 0 & -\dfrac1{\tau_D} & \dfrac1{\epsilon_\Delta\tau_D}
\end{pmatrix}.
\end{aligned}
\end{equation}
The matrices $A_V$ and $B_V$ are real symmetric, and $A_V$ is uniformly positive definite by \eqref{eq:modH} and \eqref{eq:material-positivity}. \eqref{eq:modAB} is the homogeneous Maxwell-Debye equation.

Thanks to the transverse polarization, \eqref{eq:modAB} is independent of $x_3$, hence we can consider the $(x_1,x_2)$-section where $\bm{z}_\star$ lives. We call $\mathcal A$ the annular section of the thin shell surrounding inhomogeneous medium. Figure \ref{fig:inhomogeneous-medium} briefly illustrate the system.

\begin{figure}[htbp]
    \centering
    \includegraphics[width=0.5\textwidth]{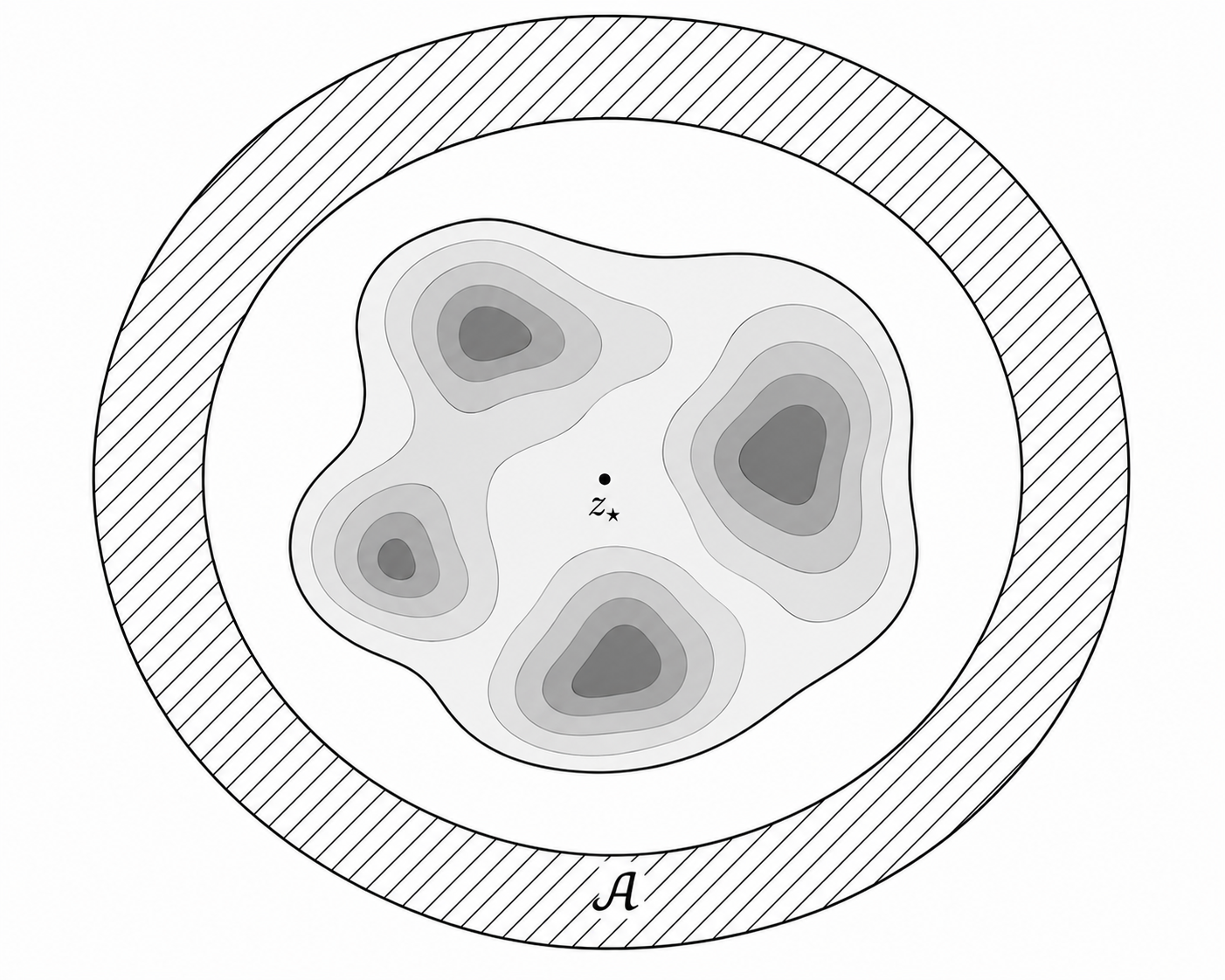}
    \caption{Schematic illustration of the inhomogeneous medium and the annulus $\mathcal{A}$.}
    \label{fig:inhomogeneous-medium}
\end{figure}

In this paper, we consider the following source term:
\begin{equation*}
    \bm J_h^{\rm ext}(t,\bm{x}) = g_h(t)\delta(\bm{x}-\bm{z}_\star)\bm{e}_z, \qquad \bm{e}_z=(0,0,1)^\top.
\end{equation*}
Such choice meets the requirement of TM fields. We choose the time origin at the end of the source pulse. Fix $\zeta_{\rm src}>0$ and let $t_{{\rm src},h}:=-h\zeta_{\rm src}$ be the center of pulse duration. Let $f\in C_c^\infty(\mathbb R;\mathbb R)$ be a nonzero and odd function, with support $\supp f\subset (-\zeta_{\rm src}, \zeta_{\rm src})$.
Then
\begin{equation*}
    g_h(t):=\frac{p}{h}f\left(\frac{t}{h}-\zeta_{\rm src}\right), \qquad p\in\mathbb C\setminus\{0\}
\end{equation*}
is an odd (generally complex-valued) pulse centered at $t_{{\rm src},h}$, of duration $O(h)$, and supported in $t<0$. As $h\to 0$, $g_h(t)$ tends to $p\delta(t-t_{{\rm src},h})$. The Maxwell-Debye equation with source $\bm J_h^{\rm ext}$ is
\begin{equation*}
    A_V\partial_t\bm{W}_h+\Lambda\bm{W}_h+B_V\bm{W}_h = -g_h(t)\delta(\bm{x}-\bm{z}_\star)\bm{e}_3, \qquad \bm{e}_3=(0,0,1,0)^\top.
\end{equation*} 

\subsection{Duhamel formula and the background propagator}

In this subsection, we demonstrate the evolution group and mild solution of \eqref{eq:modmatrix} with initial data $\bm{W}(s)=\bm{w}_s$. The results are not the crucial part of this paper, and can be achieved with standard knowledge of evolution equations following similar methods as \cite{BalFinkPinaud2019}, so we will leave out the detailed explanations and proofs. 

Let $\mathcal H_{\rm s}:=L^2(\mathbb R^2;\mathbb C)$ and $\mathcal H_{\rm v}:=L^2(\mathbb R^2;\mathbb C^4)$, and set $Y:=\{\bm{u}\in\mathcal H_{\rm v}:\Lambda \bm{u}\in\mathcal H_{\rm v}\}$. For each fixed $t$, define
\begin{equation}\label{eq:movingnorm}
(\bm{u},\bm{v})_t:=\int_{\mathbb R^2}\bm{v}(\bm{x})^*A_V(t,\bm{x})\bm{u}(\bm{x})\dd \bm{x},
\qquad
\|\bm{u}\|_t^2:=(\bm{u},\bm{u})_t.
\end{equation}
Here $\bm{v}^*$ means the conjugate transpose of $\bm{v}$. $(\cdot,\cdot)_t$ is an inner product in $\mathcal H_{\rm v}$. Write $\mathcal S(t):=-A_V(t)^{-1}\Lambda, \mathcal R(t):=-A_V(t)^{-1}B_V(t)$.

\begin{lemma}\label{lem:skew} 
For each fixed $t$, $\mathcal S(t)$ is skew-adjoint on $(\mathcal H_{\rm v},(\cdot,\cdot)_t)$ with domain $Y$. Moreover, $C_c^\infty(\mathbb R^2;\mathbb C^4)$ is a core for $\mathcal S(t)$. 
\end{lemma}

\begin{proof} 
Write $\Lambda=D_1\partial_{x_1}+D_2\partial_{x_2}$ with the constant real symmetric matrices 
\begin{equation*}
D_1=
\begin{pmatrix}
0 & 0 & 0 & 0\\
0 & 0 & -1 & 0\\
0 & -1 & 0 & 0\\
0 & 0 & 0 & 0
\end{pmatrix},
\qquad
D_2=
\begin{pmatrix}
0 & 0 & 1 & 0\\
0 & 0 & 0 & 0\\
1 & 0 & 0 & 0\\
0 & 0 & 0 & 0
\end{pmatrix}.
\end{equation*}
Integration by parts gives
$$
(\mathcal S(t)\bm{u},\bm{v})_t+(\bm{u},\mathcal S(t)\bm{v})_t = -\int_{\mathbb R^2}\sum_{j=1}^2 \partial_{x_j}(\bm{v}^*D_j\bm{u})\dd \bm{x} = 0
$$
for compactly supported smooth $\bm{u},\bm{v}$. Cutoff and mollification show that this space is a core in the graph norm of $\Lambda$. Another integration by parts gives $D(\mathcal S(t)^*)=Y$ and $\mathcal S(t)^*=-\mathcal S(t)$. 
\end{proof}

\begin{theorem} 
Under \eqref{eq:modH} and \eqref{eq:material-positivity}, the equation \eqref{eq:modmatrix} has a unique strongly continuous evolution family $U(t,s)$ on $\mathcal H_{\rm v}$. It leaves $Y$ invariant and satisfies
\begin{equation}\label{eq:modUbound}
\|U(t,s)\|_{\mathcal B(\mathcal H_{\rm v})}\le C_1\exp(C_2|t-s|)
\end{equation}
with some constants $C_1, C_2$.
\end{theorem}

\begin{proof} 
Lemma~\ref{lem:skew} gives the skew-adjoint principal part in the inner product. The operator $\mathcal R(t)$ is bounded on $\mathcal H_{\rm v}$, the domain $Y$ is independent of $t$, and $t\mapsto \mathcal{S}(t)+\mathcal{R}(t)$ is continuous from $Y$ to $\mathcal H_{\rm v}$. The derivative of the norm in \eqref{eq:movingnorm} is bounded by $C\|\bm{u}\|_t^2$. Kato's theorem for hyperbolic evolution equations with a time-dependent equivalent norm \cite{Kato1970,Kato1973} gives the evolution family and the bound \eqref{eq:modUbound}. 
\end{proof}

\begin{proposition}
Let $\bm{w}_s\in\mathcal H_{\rm v}$ and $\bm{f}\in L^1_{\rm loc}([s,\infty);\mathcal H_{\rm v})$. The unique mild solution of
$$
A_V\partial_t\bm{W}+\Lambda \bm{W}+B_V\bm{W}=\bm{f}, \qquad \bm{W}(s)=\bm{w}_s,
$$
is
\begin{equation}\label{eq:duhamel}
\bm W(t)=U(t,s)\bm{w}_s + \int_s^t U(t,r)A_V(r)^{-1}\bm{f}(r)\dd r.
\end{equation}
We call \eqref{eq:duhamel} Duhamel's formula.
\end{proposition}

\begin{proof} 
Differentiate the right-hand side for data in $Y$ and smooth $\bm{f}$, then use Lemma \ref{lem:skew} and the bound \eqref{eq:modUbound}. 
\end{proof}

For the case of $V_m=0$, i.e. the background system without modulation, \eqref{eq:modmatrix} gives the time-independent background equation
\begin{equation*}
    L_0\bm{W}:=A_0(\bm{x})\partial_t \bm{W}+\Lambda \bm{W}+B_0(\bm{x})\bm{W}=0,
\end{equation*}
where, by \eqref{eq:modAB},
\begin{equation}\label{eq:bgAB}
\begin{aligned}
&A_0=\diag\left(\mu,\mu,\epsilon_0\epsilon_\infty,\frac{1}{\epsilon_{\Delta,0}}\right),
&&
\epsilon_{\Delta,0}:=\epsilon_0\Delta\epsilon, \\
&B_0=
\begin{pmatrix}
0 & 0 & 0 & 0\\
0 & 0 & 0 & 0\\
0 & 0 & \sigma+\dfrac{\epsilon_{\Delta,0}}{\tau_D} & -\dfrac1{\tau_D}\\
0 & 0 & -\dfrac1{\tau_D} & \dfrac1{\epsilon_{\Delta,0}\tau_D}
\end{pmatrix},
&&
L_0:=A_0\partial_t+\Lambda+B_0.
\end{aligned}
\end{equation}
We also define
\begin{equation}\label{eq:q0-definition}
q_0(\bm{x}):=\frac{\sigma(\bm{x})}{\epsilon_0\epsilon_\infty(\bm{x})} + \frac{\Delta\epsilon(\bm{x})}{\epsilon_\infty(\bm{x})\tau_D(\bm{x})}.
\end{equation}
Similar as the case of general $V_m$, assuming the starting time $s=0$, the background evolution is a group, denoted by $U_0(t)$. For initial data $\bm{w}_0\in\mathcal H_{\rm v}$, the background solution is
\begin{equation}\label{eq:unpert}
\bm{W}_0(t):=U_0(t)\bm{w}_0, \qquad \bm{W}_0(0)=\bm{w}_0.
\end{equation}

\subsection{Reversed component produced by material pulse}

Now we derive the decomposition formula of total field after the modulation occurrence. By rearranging terms, \eqref{eq:modmatrix} is rewritten as
\begin{equation*}
    L_0\bm{W} = -(A_V-A_0)(t)\partial_t\bm{W} - (B_V-B_0)(t)\bm{W}
\end{equation*}
Direct calculation implies that the third component is the only non-zero term. For a scalar function $u=u(t,\bm{x})$, we define
\begin{equation*}
\bm{\mathcal F}_m[u] :=-\epsilon_0\epsilon_\infty \left[\partial_t(V_mu)+q_0V_mu\right]\bm{e}_3 +\frac1{\tau_D}V_mu \bm{e}_4,
\end{equation*}
where $q_0$ is defined in \eqref{eq:q0-definition}. Let $M_m$ denote multiplication by $m$. Then
\begin{equation*}
L_0\bm{W}= \bm{\mathcal F}_m[E_z].
\end{equation*}

Define
\begin{equation*}
\bm{\Upsilon}_0:=\partial_t\bm{W}_0(0) =-A_0^{-1}(\Lambda+B_0)\bm{w}_0.
\end{equation*}

\begin{proposition}\label{prop:backpropagating}
For $t\geq T+\tau_h/2$, define
\begin{equation}\label{eq:wR-tailored}
\bm{w}_R^m(t):=\frac12\int_0^t \eta_{\tau_h}(r-T)U_0(t-r) M_m\Gamma\partial_r\bm{W}_0(r)\dd r.
\end{equation}
Let $\Gamma:=\diag(-1,-1,1,1)$, and $E_z^0$ be the third component of $\bm{W}_0$. Then the physical solution has the exact form
\begin{equation}\label{eq:physical-field-reversed-split}
\bm{W}(t)=\bm{W}_0(t)+\bm{w}_R^m(t)+\bm{N}_m(t),
\end{equation}
where
\begin{align}\label{eq:physical-remainder-definition}
\bm{N}_m(t):=&-\frac12\int_0^t U_0(t-r)V_m(r)\partial_r\bm{W}_0(r)\dd r\nonumber\\
&+\frac1\mu\int_0^t\eta_{\tau_h}(r-T)U_0(t-r)
\begin{pmatrix}
(\partial_{x_2}m)E_z^0(r)\\
-(\partial_{x_1}m)E_z^0(r)\\
0\\0
\end{pmatrix}\dd r\nonumber\\
&+\int_0^tU_0(t-r)A_0^{-1}
\bm{\mathcal F}_m[E_z^m-E_z^0](r)\dd r.
\end{align}
At time $2T$, the reversed component is
\begin{equation}\label{eq:physical-reversed-operator}
\begin{aligned}
\bm{w}_R^m(2T) = &\frac12\int_{-\tau_h/2}^{\tau_h/2}\eta_{\tau_h}(s) U_0(T-s)M_m \Gamma U_0(T+s)\bm{\Upsilon}_0\dd s.
\end{aligned}
\end{equation}
We denote the operator in \eqref{eq:physical-reversed-operator} by
\begin{equation}\label{eq:Kmdef}
K_{m,\tau_h}\bm{\Upsilon}_0:=\bm{w}_R^m(2T).
\end{equation}
\end{proposition}

\begin{proof} 
Duhamel's formula \eqref{eq:duhamel} gives
$$
\bm{W}(t)=\bm{W}_0(t)+\int_0^tU_0(t-r)A_0^{-1} \bm{\mathcal F}_m[E_z^m](r)\dd r.
$$
Split $E_z^m=E_z^0+(E_z^m-E_z^0)$. The second part is the last line of \eqref{eq:physical-remainder-definition}. For the first part,
$$
A_0^{-1}\bm{\mathcal F}_m[E_z^0] =-\partial_r(V_mE_z^0)\bm{e}_3-q_0V_mE_z^0\bm{e}_3 +\frac{\epsilon_{\Delta,0}}{\tau_D}V_mE_z^0\bm{e}_4.
$$
Integrating the time derivative by parts and using the explicit matrices $\Lambda$ and $B_0$ gives
$$
\int_0^tU_0(t-r)A_0^{-1}\bm{\mathcal F}_m[E_z^0](r)\dd r = \frac1\mu\int_0^tU_0(t-r)\nabla^\perp(V_mE_z^0)(r)\dd r,
$$
where $\nabla^\perp:=(\partial_{x_2}, -\partial_{x_1})$. From the first two equations in \eqref{eq:system}, 
$$
\frac1\mu V_m\nabla^\perp E_z^0=-V_m\partial_r\bm{W}_H, \qquad \bm{W}_H:=\frac12(I-\Gamma)\bm{W}_0,
$$
one obtains
$$
\frac1\mu V_m\nabla^\perp E_z^0=-\frac12V_m\partial_r\bm{W}_0+\frac12V_m\Gamma\partial_r\bm{W}_0.
$$
The second term on the right-hand side gives \eqref{eq:wR-tailored}. Collecting the remaining terms gives \eqref{eq:physical-remainder-definition}. This proves \eqref{eq:physical-field-reversed-split}.

Because the background coefficients are independent of time, $\partial_r\bm{W}_0(r)=U_0(r)\bm{\Upsilon}_0$. In \eqref{eq:wR-tailored}, setting $r=T+s$ gives \eqref{eq:physical-reversed-operator}. 
\end{proof}

\section{Semiclassical Mode Reduction and the Returned-Wave Symbol}\label{sec:semiclassical}

Proposition~\ref{prop:backpropagating} expresses the reversed component generated by the material pulse in terms of the full Maxwell--Debye propagator $U_0(t)$. To analyze this operator in the high-frequency regime, we now reduce the propagating TM modes of the vector system to scalar semiclassical dynamics. The purpose of this section is to obtain a microlocal scalar representation of the returned field and, ultimately, an explicit formula for the principal symbol of the returned operator. This formula will provide the main input for the modulation construction and refocusing analysis in the subsequent sections.

We begin with the positive propagating TM branch and identify its dispersion relation $\omega$, normalized polarization $\bm r$, Hamiltonian flow $\bm\Phi^t$, and leading material attenuation $\gamma$. The corresponding injection and extraction symbols are constructed, so that the corresponding scalar evolution intertwines with the full Maxwell--Debye dynamics modulo negligible errors. We also factor the positive-branch scalar propagator into Hamiltonian transport along $\bm\Phi^t$ and an accumulated damping factor. The returned field involves the involution $\Gamma$, which maps the positive TM polarization to the negative one. Repeating the scalar reduction on the negative branch gives the corresponding coefficient $\beta_-=\gamma-\kappa$. Using these two scalar propagators, we compress the physical returned operator to a scalar operator and compute its principal symbol for pulse width $\tau_h=\nu h$. The resulting symbol separates the effects of the temporal pulse response, the spatial modulation evaluated at the outgoing arrival point, and the attenuation accumulated during the round trip. We conclude by recording estimates uniform for $h$-dependent mirror times and modulation functions. The semiclassical notation and conventions used throughout are collected in Appendix~\ref{sec:semi-convention}.

\subsection{The positive TM branch and its Hamiltonian flow}

Let $0<h\ll1$ denote the wavelength scale. We look for a high-frequency solution of the background equation in the form
\begin{equation}\label{eq:wkb-ansatz}
    \bm W(t,\bm x)=\bm a_h(t,\bm x)e^{\ii\phi(t,\bm x)/h},
\end{equation}
where $\bm a_h$ is a classical amplitude,
\begin{equation*}
    \bm a_h\sim\sum_{j=0}^{\infty}h^j\bm a_j.
\end{equation*}
Here the asymptotic notation means that for every compact $K_{t,\bm x}$, every $k,N\in\mathbb N$, and every differential operator $D_{t,\bm x}$ of order at most $k$,
\begin{equation*}
 \sup_{(t,\bm x)\in K_{t,\bm x}}\left|D_{t,\bm x}\left(\bm a_h-\sum_{i=0}^{N-1}h^i\bm a_i\right)\right|\le C_{K_{t,\bm x},k,N}h^N.
\end{equation*}

We are aimed at solving $\bm W$ asymptotically. To this end, we need to asymptotically solve $\bm a_h$ and $\phi$. Set $\bm\xi:=\nabla_{\bm x}\phi$. Substitution of \eqref{eq:wkb-ansatz} into $L_0\bm W=0$ gives
\begin{equation*}
\begin{aligned}
0=e^{\ii\phi/h}\bigg(&\frac{\ii}{h}\left[(\partial_t\phi)A_0(\bm x)+\widetilde\Lambda(\bm\xi)\right]\bm a_h+A_0\partial_t\bm a_h+\Lambda\bm a_h+B_0\bm a_h\bigg),
\end{aligned}
\end{equation*}
where
\begin{equation*}
    \widetilde\Lambda(\bm\xi)
    =D_1\xi_1+D_2\xi_2
    =\begin{pmatrix}
        0 & 0 & \xi_2 & 0\\
        0 & 0 & -\xi_1 & 0\\
        \xi_2 & -\xi_1 & 0 & 0\\
        0 & 0 & 0 & 0
      \end{pmatrix}.
\end{equation*}
The coefficient of $h^{-1}$ determines the principal polarization. Writing $\omega=-\partial_t\phi$, we must solve
\begin{equation}\label{eq:leadingorder}
    \widetilde\Lambda(\bm\xi)\bm r(\bm x,\bm\xi) = \omega(\bm x,\bm\xi)A_0(\bm x)\bm r(\bm x,\bm\xi).
\end{equation}

If we write $\bm r=(h_1,h_2,e,p_{\rm pol})^\top$, equation \eqref{eq:leadingorder} then becomes
\begin{equation*}
    \xi_2e=\omega\mu h_1,\qquad
    -\xi_1e=\omega\mu h_2,\qquad
    \xi_2h_1-\xi_1h_2=\omega\epsilon_0\epsilon_\infty e,
    \qquad
    0=\omega\epsilon_{\Delta,0}^{-1}p_{\rm pol}.
\end{equation*}
For propagating branch, $\omega\neq 0$, so we have $p_{\rm pol}=0$. Here we firstly solve positive propagating branch, which requires $\omega>0$. Eliminating $h_1$ and $h_2$ then gives
\begin{equation}\label{eq:positive-dispersion}
    \omega(\bm x,\bm\xi)
    =\frac{|\bm\xi|}{\sqrt{\mu\epsilon_0\epsilon_\infty(\bm x)}}.
\end{equation}
Choosing the $A_0$-normalized eigenvector gives
\begin{equation}\label{eq:explicit-r-positive-branch}
    \bm r(\bm x,\bm\xi)
    =\left(\frac{\widehat\xi_2}{\sqrt{2\mu}}, -\frac{\widehat\xi_1}{\sqrt{2\mu}}, \frac{1}{\sqrt{2\epsilon_0\epsilon_\infty(\bm x)}}, 0\right)^\top, \qquad
    \widehat{\bm\xi}:=\frac{\bm\xi}{|\bm\xi|},
\end{equation}
so that
\begin{equation*}
    \bm r(\bm x,\bm\xi)^*A_0(\bm x)\bm r(\bm x,\bm\xi)=1.
\end{equation*}

Let $\mathcal U\subset T^*\mathbb R^2\setminus\{0\}$ be an open conic set on which the branch is defined. The functions $\omega$ and $\bm r$ are smooth on $\mathcal U$, $\omega$ is homogeneous of degree one in $\bm\xi$, and the eigenvalue $\omega$ is simple. Hence the leading amplitude on this branch has the form
\begin{equation*}
    \bm a_0(t,\bm x)=a_0(t,\bm x)\bm r(\bm x,\nabla_{\bm x}\phi(t,\bm x)).
\end{equation*}

Now we calculate $\bm{x}(t)$ and $\bm{\xi}(t)$. From the definition of $\omega$ we observe that
\begin{align*}
    \bm{\xi}'(t) &= \frac{\mathrm{d}}{\mathrm{d}t}\nabla_{\bm{x}}\phi(t,\bm{x}(t)) = \nabla_{\bm{x}}\partial_t\phi(t,\bm{x}(t)) + D_{\bm{x}}^2\phi(t,\bm{x}(t))\bm{x}'(t), \\
    &= -\nabla_{\bm{x}}\omega(\bm{x}(t),\bm{\xi}(t)) - D_{\bm{x}}^2\phi(t,\bm{x}(t))\partial_{\bm{\xi}}\omega(\bm{x}(t),\bm{\xi}(t))+D_{\bm{x}}^2\phi(t,\bm{x}(t))\bm{x}'(t).
\end{align*}
Let $\bm{x}(t)$ solve the equation $\bm{x}'(t)=\partial_{\bm{\xi}}\omega(\bm{x}(t),\bm{\xi}(t))$, we then get 
$$
\bm{\xi}'(t)=-\partial_{\bm{x}}\omega(\bm{x}(t),\bm{\xi}(t)).
$$
As a result, if we set $\bm\rho:=(\bm x,\bm\xi)$ and denote by $\bm\Phi^t(\bm\rho)$ the flow starting from $\bm\rho$ and evolving for time $t$, then
\begin{equation*}
    \frac{\dd}{\dd t}\bm{\Phi}^t(\bm{\rho})=\bm H_\omega(\bm{\Phi}^t(\bm{\rho})), \qquad \bm{\Phi}^0(\bm{\rho})=\bm{\rho}, 
\end{equation*}
where 
\begin{equation*}
    \bm H_\omega = \sum_{j=1}^2\left(\partial_{\xi_j}\omega\partial_{x_j}   -\partial_{x_j}\omega\partial_{\xi_j}\right)
\end{equation*}
denotes the Hamilton vector field of $\omega$. The field implies that a positive-branch wave packet centered at $\bm\rho$ moves, to leading order, along $\bm\Phi^t(\bm\rho)$. $\phi$ could be immediately solved with $\omega(\bm{x},\bm{\xi})=-\partial_t\phi$ after $\bm{x}(t)$ and $\bm{\xi}(t)$ are solved. 

The matrix $B_0$ enters one order below the principal spatial operator. Its contribution on the selected mode is obtained by testing against the normalized left mode
\begin{equation*}
    \bm\ell(\bm x,\bm\xi):=\bm r(\bm x,\bm\xi)^*A_0(\bm x),
    \qquad
    \bm\ell\bm r=1.
\end{equation*}
The resulting material damping coefficient is
\begin{equation*}
    \gamma(\bm x,\bm\xi):=\bm r(\bm x,\bm\xi)^*B_0(\bm x)\bm r(\bm x,\bm\xi).
\end{equation*}
It is nonnegative because
\begin{equation*}
    \bm u^*B_0(\bm x)\bm u=\sigma(\bm x)|u_3|^2+\frac{|\epsilon_{\Delta,0}(\bm x)u_3-u_4|^2}{\epsilon_{\Delta,0}(\bm x)\tau_D(\bm x)}\ge0.
\end{equation*}
Using \eqref{eq:bgAB} and \eqref{eq:explicit-r-positive-branch},
\begin{equation*}
    \gamma(\bm x,\bm\xi)=\frac12\left(\frac{\sigma(\bm x)}{\epsilon_0\epsilon_\infty(\bm x)}+\frac{\Delta\epsilon(\bm x)}{\epsilon_\infty(\bm x)\tau_D(\bm x)}\right)
    =\frac12q_0(\bm x).
\end{equation*}

The coefficient $\gamma$ has a direct physical interpretation. It is obtained by projecting the dissipative part $B_0$ of the Maxwell--Debye system onto the normalized positive-TM polarization. Thus $\gamma$ represents the leading \emph{material attenuation rate} along the selected branch. In the present model its two contributions come respectively from Ohmic conduction and Debye relaxation:
$$
    \gamma(\bm x,\bm\xi)=\frac12\frac{\sigma(\bm x)}{\epsilon_0\epsilon_\infty(\bm x)}+\frac12\frac{\Delta\epsilon(\bm x)}{\epsilon_\infty(\bm x)\tau_D(\bm x)}.
$$
In particular, $\gamma\ge0$, and the material contribution to the amplitude accumulated along a ray segment $\bm\Phi^s(\bm\rho)$ is
$$
    \exp\left(-\int_0^t\gamma(\bm\Phi^s(\bm\rho))\dd s\right).
$$
This material attenuation should be distinguished from the geometric transport correction introduced below, which arises from the phase-space variation of the polarization rather than from $B_0$.

\subsection{Constructing the scalar branch dynamics}

The previous calculation gives the leading polarization $\bm r$ and the leading material loss $\gamma$. This is not yet enough to reduce the vector equation, because $\bm r$ varies in phase space, applying the vector generator to $\Op_h(\bm r)u$ also produces order-zero terms in the other modes. We now correct the mode symbol and determine the scalar lower-order term so that the scalar and vector evolutions agree to all orders in $h$.

The background equation has the semiclassical form
\begin{equation}\label{eq:semiclassical-system}
    \left(\partial_t+\frac{\ii}{h}\Op_h(A_0^{-1}\widetilde\Lambda)+\Op_h(A_0^{-1}B_0)\right)\bm W=0.
\end{equation}
Set
\begin{equation*}
    \mathsf M_{\rm md}(\bm x,\bm\xi):=A_0(\bm x)^{-1}\widetilde\Lambda(\bm\xi),
    \qquad
    \mathsf C_{\rm md}(\bm x):=A_0(\bm x)^{-1}B_0(\bm x).
\end{equation*}
On $\mathcal U$, the symbols $\mathsf M_{\rm md}$ and $\omega$ have order $1$, while $\mathsf C_{\rm md}$, $\bm r$, and $\bm\ell$ have order $0$. Moreover,
\begin{equation*}
    \mathsf M_{\rm md}\bm r=\omega\bm r,
    \qquad
    \bm\ell \mathsf M_{\rm md}=\omega\bm\ell,
    \qquad
    \bm\ell \mathsf C_{\rm md}\bm r=\gamma.
\end{equation*}
Define
\begin{equation*}
    \Pi:=\bm r\bm\ell,
    \qquad
    Q:=I_4-\Pi.
\end{equation*}
The projection $\Pi$ selects the positive mode, while $Q$ selects the other modes. Both projections have order $0$.

Fix $T>0$ and a compact set $K\Subset\mathcal U$. Choose $\delta_T>0$ such that
\begin{equation*}
    K_{T,\delta_T}:=\left\{\bm\Phi^s(\bm\rho):\bm\rho\in K, |s|\le T+\delta_T\right\}\Subset\mathcal U.
\end{equation*}
Since $\omega$ is simple, its distance from the remaining spectrum of $\mathsf M_{\rm md}$ is bounded below on this compact flow tube. Hence
\begin{equation*}
    Q(\mathsf M_{\rm md}-\omega)Q:\operatorname{Ran}Q\longrightarrow\operatorname{Ran}Q
\end{equation*}
is invertible, and its inverse is a symbol of order $-1$. All constructions below are made on this flow tube and are then localized by fixed microlocal cutoffs.

For the left Kohn--Nirenberg quantization, the symbol product satisfies
\begin{equation}\label{eq:symbol-product}
    a_h\# b_h = a_hb_h+\frac{h}{\ii}\sum_{j=1}^2(\partial_{\xi_j}a_h)(\partial_{x_j}b_h)+h^2r_{a,b,h},
\end{equation}
where $r_{a,b,h}$ has order two lower than the product $a_hb_h$.

\subsubsection{Construction of injection symbol.}\label{sec:injection-symbol}
Suppose that a scalar amplitude $u$ is inserted into the vector system through a symbol $\bm r_h$. We want to find a symbol $\bm r_h$, such that a scalar amplitude $u$, applied by $\Op_h(\bm{r}_h)$, is the vector solution of \eqref{eq:semiclassical-system}, and 
\begin{equation}\label{eq:intertwine}
\left(\partial_t+\frac{\ii}{h}\Op_h(\mathsf M_{\rm md})+\Op_h(\mathsf C_{\rm md})\right)\Op_h(\bm r_h)u=\Op_h(\bm r_h)\left(\partial_t+\frac{\ii}{h}\Op_h(\omega)+\Op_h(\beta_h)\right)u \quad\text{modulo }O(h^\infty).
\end{equation}
Because the symbols are independent of time, this requirement is equivalent to
\begin{equation}\label{eq:right-symbolic-intertwining}
    \mathsf M_{\rm md}\#\bm r_h-\bm r_h\#\omega-\ii h\left(\mathsf C_{\rm md}\#\bm r_h-\bm r_h\#\beta_h\right)\sim0.
\end{equation}
We seek formal expansions
\begin{equation*}
    \bm r_h\sim\bm r+h\bm r_1+h^2\bm r_2+\cdots,
    \qquad
    \beta_h\sim\beta+h\beta_1+h^2\beta_2+\cdots.
\end{equation*}
The term of order $h^0$ in \eqref{eq:right-symbolic-intertwining} is exactly the eigenvalue equation $\mathsf M_{\rm md}\bm r=\omega\bm r$.

At order $h$, define
\begin{equation*}
    \bm{\mathcal D}:=\sum_{j=1}^2\left[(\partial_{\xi_j}\mathsf M_{\rm md})(\partial_{x_j}\bm r)-(\partial_{\xi_j}\bm r)(\partial_{x_j}\omega)\right].
\end{equation*}
Matching order $h$ of \eqref{eq:right-symbolic-intertwining} yields
\begin{equation}\label{eq:first-right-correction}
    (\mathsf M_{\rm md}-\omega)\bm r_1 = \ii\left(\bm{\mathcal D}+\mathsf C_{\rm md}\bm r-\bm r\beta\right).
\end{equation}
The left side has no component in the selected eigenspace. Therefore the right side must also have no such component. Since $\bm\ell(\mathsf M_{\rm md}-\omega)=0$, the right-hand side of \eqref{eq:first-right-correction} should be 0 as well when applied by $\bm\ell$, i.e.
$$
    0 = \bm{\ell}\bm{\mathcal D}+\gamma-\beta.
$$

Define
\begin{equation*}
    \kappa:=\bm\ell\bm{\mathcal D}.
\end{equation*}
Then the solvability condition gives
\begin{equation}\label{eq:beta-material-geometric-split}
    \beta=\gamma+\kappa.
\end{equation}
The two terms in \eqref{eq:beta-material-geometric-split} have different origins. As explained above, $\gamma=\bm r^*B_0\bm r$ is the material attenuation rate. In contrast, $\kappa$ does not arise from the dissipative matrix $B_0$. Tt is generated by the phase-space variation of the principal symbol and of the selected polarization $\bm r$:
$$
    \kappa=\bm\ell\sum_{j=1}^2\left[(\partial_{\xi_j}\mathsf M_{\rm md})(\partial_{x_j}\bm r)-(\partial_{\xi_j}\bm r)(\partial_{x_j}\omega)\right],
$$
We therefore refer to $\kappa$ as the \emph{geometric transport correction}. It describes the order-zero change of the scalar amplitude associated with transport of the varying eigenspace along the Hamiltonian flow. Thus $\beta$ is the total leading amplitude-transport coefficient on the positive branch, consisting of the physical material contribution $\gamma$ and the geometric contribution $\kappa$.

The vector $\bm{\mathcal D}+\mathsf C_{\rm md}\bm r-\bm r\beta$ then lies in $\operatorname{Ran}Q$, where $\mathsf M_{\rm md}-\omega$ is invertible. We therefore choose
\begin{equation*}
    \bm r_1:=\ii\bigl(Q(\mathsf M_{\rm md}-\omega)Q\bigr)^{-1}Q\bigl(\bm{\mathcal D}+\mathsf C_{\rm md}\bm r-\bm r\beta\bigr),  \qquad\bm\ell\bm r_1=0.
\end{equation*}
The inverse has order $-1$, so $\bm r_1$ has order $-1$.

The same two steps determine all higher coefficients. Notice that such construction means $\bm{r}_j\in\operatorname{Ran} Q$, so $\bm\ell\bm r_j=0$ for $j\ge1$.

\subsubsection{Construction of the extraction symbol.}\label{sec:extraction-symbol}
The injection symbol inserts a scalar amplitude into the vector system. To recover that amplitude, we seek a row symbol
\begin{equation*}
    \bm\ell_h\sim\bm\ell+h\bm\ell_1+h^2\bm\ell_2+\cdots
\end{equation*}
with two properties:
\begin{equation}\label{eq:left-symbolic-intertwining}
    \bm\ell_h\#\mathsf M_{\rm md}-\omega\#\bm\ell_h-\ii h\left(\bm\ell_h\#\mathsf C_{\rm md}-\beta_h\#\bm\ell_h\right)\sim0,
\end{equation}
and
\begin{equation}\label{eq:symbolic-normalization}
    \bm\ell_h\#\bm r_h\sim1.
\end{equation}
The first relation makes extraction intertwine with the evolution like \eqref{eq:intertwine}, and the second relation makes extraction undo injection.

At order $h$, \eqref{eq:left-symbolic-intertwining} gives
\begin{equation*}
    \bm\ell_1(\mathsf M_{\rm md}-\omega)=\ii\bigg(\sum_{j=1}^2\left[(\partial_{\xi_j}\bm\ell)(\partial_{x_j}\mathsf M_{\rm md})-(\partial_{\xi_j}\omega)(\partial_{x_j}\bm\ell)\right]+\bm\ell \mathsf C_{\rm md}-\beta\bm\ell\bigg).
\end{equation*}
Differentiating $\mathsf M_{\rm md}\bm r=\omega\bm r$, $\bm\ell \mathsf M_{\rm md}=\omega\bm\ell$, and $\bm\ell\bm r=1$ shows that 
$$
\sum_{j=1}^2\left[(\partial_{\xi_j}\bm\ell)(\partial_{x_j}\mathsf M_{\rm md})-(\partial_{\xi_j}\omega)(\partial_{x_j}\bm\ell)\right]\bm{r} = \bm\ell\bm{\mathcal D}.
$$
Since $(\mathsf M_{\rm md}-\omega)\bm{r}=0$, it follows $\bm\ell_1(\mathsf M_{\rm md}-\omega)\bm{r}=0$, so 
$$
\ii\bigg(\sum_{j=1}^2\left[(\partial_{\xi_j}\bm\ell)(\partial_{x_j}\mathsf M_{\rm md})-(\partial_{\xi_j}\omega)(\partial_{x_j}\bm\ell)\right]+\bm\ell \mathsf C_{\rm md}-\beta\bm\ell\bigg)\bm{r}=0.
$$
The two equations above again yield the condition $\beta=\gamma+\bm\ell\bm{\mathcal D}$.

The complementary part $\bm\ell_1Q$ is determined by
\begin{equation}\label{eq:l1q}
    \bm\ell_1Q=\ii\bigg(\sum_{j=1}^2\left[(\partial_{\xi_j}\bm\ell)(\partial_{x_j}\mathsf M_{\rm md})-(\partial_{\xi_j}\omega)(\partial_{x_j}\bm\ell)\right]+\bm\ell \mathsf C_{\rm md}-\beta\bm\ell\bigg)\bigl(Q(\mathsf M_{\rm md}-\omega)Q\bigr)^{-1}.
\end{equation}
The component along $\bm\ell$ is fixed by the normalization
\eqref{eq:symbolic-normalization}. Since $\bm\ell\bm r_1=0$, its first-order part is
\begin{equation}\label{eq:l1r}
    \bm\ell_1\bm r=-\frac1\ii\sum_{j=1}^2(\partial_{\xi_j}\bm\ell)(\partial_{x_j}\bm r).
\end{equation}
\eqref{eq:l1q} and \eqref{eq:l1r} determine $\bm\ell_1$, which has order $-1$. At every higher order, the left intertwining equation determines $\bm\ell_jQ$, and the normalization determines $\bm\ell_j\bm r$.

To conclude, the three formal series gives classical symbols
\begin{equation}\label{eq:full-mode-symbols}
\begin{aligned}
    \bm r_h&=\bm r+h\bm r_1+h^2\widetilde{\bm r}_h,\\
    \bm\ell_h&=\bm\ell+h\bm\ell_1+h^2\widetilde{\bm\ell}_h,\\
    \beta_h&=\beta+h\widetilde\beta_h.
\end{aligned}
\end{equation}
Here $\widetilde{\bm r}_h$ and $\widetilde{\bm\ell}_h$ are uniformly of order $-2$, and $\widetilde\beta_h$ is uniformly of order $-1$.

\subsubsection{Mode injection and extraction operators.}
Define
\begin{equation*}
    J_h:=\Op_h(\bm r_h), \qquad J_h^{\#}:=\Op_h(\bm\ell_h).
\end{equation*}
The operator $J_h$ inserts a scalar amplitude into the positive vector mode, and $J_h^{\#}$ extracts that amplitude. They are uniformly bounded from $\mathcal H_{\rm s}$ to $\mathcal H_{\rm v}$ and from $\mathcal H_{\rm v}$ to $\mathcal H_{\rm s}$, respectively.

Discussions in section \ref{sec:injection-symbol} and \ref{sec:extraction-symbol} give, microlocally on $K_{T,\delta_T}$,
\begin{align}
\left(\partial_t+\frac{\ii}{h}\Op_h(\mathsf M_{\rm md})+\Op_h(\mathsf C_{\rm md})\right)J_h &= J_h\left(\partial_t+\frac{\ii}{h}\Op_h(\omega)+\Op_h(\beta_h)\right)+R_h^{\infty,+},
\label{eq:intertwine-relation}\\
J_h^{\#}\left(\partial_t+\frac{\ii}{h}\Op_h(\mathsf M_{\rm md})+\Op_h(\mathsf C_{\rm md})\right) &= \left(\partial_t+\frac{\ii}{h}\Op_h(\omega)+\Op_h(\beta_h)\right)J_h^{\#}+\widetilde R_h^{\infty,+}.
\label{eq:left-intertwine-relation}
\end{align}
where $R_h^{\infty,+}$ and $\widetilde R_h^{\infty,+}$ are negligible operators.

The normalization \eqref{eq:symbolic-normalization} gives
\begin{equation}\label{eq:Jh-identities}
    J_h^{\#}J_h=I+R_h^{\infty,{\rm n}},
    \qquad
    J_hJ_h^{\#}=\Op_h(\Pi)+hR_h^{\rm p},
\end{equation}
where, $R_h^{\infty,{\rm n}}$ is negligible, and $R_h^{\rm p}$ is uniformly of order $-1$, hence bounded from $H_h^s(\mathbb R^2;\mathbb C^4)$ to $H_h^{s+1}(\mathbb R^2;\mathbb C^4)$ for every fixed $s$, and
\begin{equation*}
    \|J_hJ_h^{\#}-\Op_h(\Pi)\|_{\mathcal H_{\rm v}\to\mathcal H_{\rm v}}
    \le Ch.
\end{equation*}

The preceding construction reduces the positive vector mode to a scalar equation. If $u(t)$ denotes the scalar amplitude, then the intertwining relation \eqref{eq:intertwine} shows that $u$ should satisfy
$$
    (\partial_t+\frac{\ii}{h}\Op_h(\omega)+\Op_h(\beta_h))u=0. 
$$

\begin{definition}
The positive-branch scalar propagator $S_h(t)$ is the strongly continuous propagator on $\mathcal H_{\rm s}$, with domain $H_h^1(\mathbb R^2)$, defined by
\begin{equation}\label{eq:positive-scalar-propagator}
    \left(\partial_t+\frac{\ii}{h}\Op_h(\omega)+\Op_h(\beta_h)\right)S_h(t)=0,
    \qquad
    S_h(0)=I.
\end{equation}
\end{definition}
\begin{lemma}\label{lem:scalar-energy-estimate}
Let $S_h(t)$ be defined as in \eqref{eq:positive-scalar-propagator}. Then, for every $s\in\mathbb R$ and every $T_0>0$,
$$
\sup_{|t|\le T_0}\|S_h(t)\|_{H_h^s\to H_h^s}\le C_{s,T_0},
$$
uniformly for small $h$.
\end{lemma}

\begin{proof}
Let $u(t)=S_h(t)u_0$. Since $\omega$ is real, semiclassical symbolic calculus gives
$$
\Op_h(\omega)^*=\Op_h(\omega)+h\Op_h(r_h),
$$
with $r_h$ uniformly of order $0$. Since $\beta_h$ is also uniformly of order $0$,
$$
\frac{\dd}{\dd t}\|u(t)\|_{L^2}^2=-2\Re\left\langle\frac{\ii}{h}\Op_h(\omega)u+\Op_h(\beta_h)u,u\right\rangle\le C\|u(t)\|_{L^2}^2,
$$
where $C$ is independent of $h$. Gronwall's inequality proves the lemma for $s=0$.

For general $s\in\mathbb R$, let $\Lambda_h^s=\langle hD\rangle^s$ and set
$v=\Lambda_h^su$. Symbolic calculus gives
$$
\Lambda_h^s\Op_h(\omega)\Lambda_h^{-s}=\Op_h(\omega)+h\Op_h(r_{s,h}), \qquad \Lambda_h^s\Op_h(\beta_h)\Lambda_h^{-s}=\Op_h(b_{s,h}),
$$
where $r_{s,h}$ and $b_{s,h}$ are uniformly of order $0$. Hence $v$ satisfies an equation of the same form,
$$
\partial_tv+\frac{\ii}{h}\Op_h(\omega)v+\Op_h(c_{s,h})v=0,
$$
with $c_{s,h}$ uniformly of order $0$. Applying the $L^2$ estimate above to $v$ yields
$$
\|u(t)\|_{H_h^s}\le C_{s,T_0}\|u_0\|_{H_h^s}, \qquad |t|\le T_0,
$$
uniformly for small $h$.
\end{proof}

With exactly the same method, we could prove an estimation for background propagator $U_0(t)$ as follows.
\begin{lemma}\label{lem:background-energy-estimate}
Let $U_0(t)$ be the propagator of the background system
$$
A_0\partial_t\bm W+\Lambda\bm W+B_0\bm W=0.
$$
Then, for every $s\in\mathbb R$ and $T_0>0$,
$$
\sup_{|t|\le T_0}\|U_0(t)\|_{H_h^s(\mathbb R^2;\mathbb C^4)\to H_h^s(\mathbb R^2;\mathbb C^4)}\le C_{s,T_0},
$$
uniformly for small $h$. 
\end{lemma}

The scalar propagator on the positive branch $S_h(t)$ relates to Maxwell-Debye propagator $U_0(t)$ in \eqref{eq:unpert} via the following theorem.
\begin{theorem}
Microlocally on the flow tube $K_{T,\delta_T}$ and uniformly for $t$ in a fixed compact interval,
\begin{align}
    U_0(t)J_h&=J_hS_h(t)+\mathcal R_h^{\infty,+}(t),
    \label{eq:intertwining-operator-remainder}\\
    J_h^{\#}U_0(t)&=S_h(t)J_h^{\#}+\widetilde{\mathcal R}_h^{\infty,+}(t),
    \label{eq:left-intertwining-operator-remainder}
\end{align}
where $\mathcal R_h^{\infty,+}$ and $\widetilde{\mathcal R}_h^{\infty,+}$ are negligible. Furthermore,
\begin{equation}\label{eq:compressed-intertwining-remainder}
    J_h^{\#}U_0(t)J_h=S_h(t)+\mathcal R_h^{\infty,{\rm c}}(t),
\end{equation}
with $\mathcal R_h^{\infty,{\rm c}}$ negligible.
\end{theorem}
\begin{proof}
Equations \eqref{eq:intertwine-relation} and \eqref{eq:positive-scalar-propagator}, followed by Duhamel's formula \eqref{eq:duhamel}, give
\begin{equation*}
U_0(t)J_h-J_hS_h(t)=-\int_0^tU_0(t-s)R_h^{\infty,+}S_h(s)\dd s.
\end{equation*}
Combining the energy estimations in Lemma \ref{lem:scalar-energy-estimate}, Lemma \ref{lem:background-energy-estimate} and negligible nature of $R_h^{\infty, +}$ gives the desired result for $\mathcal R_h^{\infty,+}$. This proves \eqref{eq:intertwining-operator-remainder}. The left relation follows in the same way from \eqref{eq:left-intertwine-relation}. Composing either relation with the normalization in \eqref{eq:Jh-identities} proves \eqref{eq:compressed-intertwining-remainder}.
\end{proof}

\subsubsection{Characterization of scalar propagator.}
$S_h(t)$ consists of two parts: the lossless transport part generated by $\omega$ and the decaying part generated by $\beta_h$. We first isolate the lossless transport. Let $F_h(t)$ be the scalar propagator on $\mathcal H_{\rm s}$ generated by $\Op_h(\omega)$ with domain $H_h^1(\mathbb R^2)$:
\begin{equation}\label{eq:lossless-propagator}
    \left(\partial_t+\frac{\ii}{h}\Op_h(\omega)\right)F_h(t)=0,
    \qquad F_h(0)=I.
\end{equation}
For $t$ in a fixed compact interval and after the fixed microlocal cutoffs, $F_h(t)$ is an order-zero propagating operator along $\bm\Phi^t$ and uniformly bounded on $H_h^s(\mathbb R^2)$ for every fixed $s\in\mathbb R$.

\begin{theorem}\label{thm:egorov}
Let $a_h$ be a compactly supported symbol of order $0$ in $\mathcal U$, and assume that its support remains compactly contained in $\mathcal U$ under the flow for the times under consideration. Then
\begin{equation}\label{eq:egorov-with-class}
    F_h(-t)\Op_h(a_h)F_h(t)=\Op_h(a_h\circ\bm\Phi^t)+hR_{a,h}(t),
\end{equation}
where $R_{a,h}(t)$ is a uniformly bounded order-zero semiclassical pseudodifferential operator satisfying
\begin{equation*}
 \sup_t\|hR_{a,h}(t)\|_{\mathcal H_{\rm s}\to\mathcal H_{\rm s}}\le Ch.
\end{equation*}
A full symbol of $R_{a,h}(t)$ can be chosen compactly supported and of order $0$, with each finite derivative bound controlled by finitely many derivative bounds for $a_h$.
\end{theorem}
\begin{proof}
Since the generator in \eqref{eq:lossless-propagator} is time independent, $F_h(t)=\exp(-\ii t\Op_h(\omega)/h)$ and it commutes with $\Op_h(\omega)$. Set
\begin{equation*}
    \mathsf B_h^{\rm Eg}(t):=F_h(-t)\Op_h(a_h)F_h(t), \qquad \mathsf C_h^{\rm Eg}(t):=\Op_h(a_h\circ\bm\Phi^t).
\end{equation*}
Then $\mathsf B_h^{\rm Eg}(0)=\mathsf C_h^{\rm Eg}(0)=\Op_h(a_h)$ and
\begin{equation*}
    \partial_t\mathsf B_h^{\rm Eg}(t)=\frac{\ii}{h}[\Op_h(\omega),\mathsf B_h^{\rm Eg}(t)], \qquad \partial_t\mathsf C_h^{\rm Eg}(t)=\Op_h\bigl(\bm H_\omega(a_h\circ\bm\Phi^t)\bigr).
\end{equation*}
From \eqref{eq:symbol-product}, we have
\begin{align*}
    \omega\#(a_h\circ\bm{\Phi}^t)-(a_h\circ\bm{\Phi}^t)\#\omega&=\frac{h}{\ii}\sum_j (\partial_{\xi_j}\omega)(\partial_{x_j}(a_h\circ\bm{\Phi}^t))-(\partial_{x_j}\omega)(\partial_{\xi_j}(a_h\circ\bm{\Phi}^t))+h^2r_h \\
    &= \frac{h}{\ii}\bm H_\omega(a_h\circ\bm{\Phi}^t)+h^2r_h.
\end{align*}

Hence there exists a uniformly bounded order-zero semiclassical pseudodifferential operator $R_h^{\rm Eg}(t)$ such that
\begin{equation*}
    \partial_t\mathsf C_h^{\rm Eg}(t)=\frac{\ii}{h}[\Op_h(\omega),\mathsf C_h^{\rm Eg}(t)]+hR_h^{\rm Eg}(t).
\end{equation*}

Thus $\mathsf D_h^{\rm Eg}:=\mathsf B_h^{\rm Eg}-\mathsf C_h^{\rm Eg}$ satisfies
\begin{equation*}
    \partial_t\mathsf D_h^{\rm Eg}=\frac{\ii}{h}[\Op_h(\omega),\mathsf D_h^{\rm Eg}(t)]-hR_h^{\rm Eg}(t), \qquad \mathsf D_h^{\rm Eg}(0)=0.
\end{equation*}

Moreover, since $F_h$ commutes with $\Op_h(\omega)$, 
\begin{align*}
    &\partial_t(F_h(t)\mathsf D_h^{\rm Eg}(t)F_h(-t)) \\ =& (\partial_tF_h(t))\mathsf D_h^{\rm Eg}(t)F_h(-t) + F_h(t)(\partial_t\mathsf D_h^{\rm Eg}(t))F_h(-t) + F_h(t)\mathsf D_h^{\rm Eg}(t)\partial_tF_h(-t) \\
    =& (-\frac{\ii}{h}\Op_h(\omega)F_h(t))\mathsf D_h^{\rm Eg}(t)F_h(-t) +F_h(t)\left(\frac{\ii}{h}[\Op_h(\omega), \mathsf D_h^{\rm Eg}(t)]-hR_h^{\rm Eg}(t)\right)F_h(-t) \\ &+ F_h(t)\mathsf D_h^{\rm Eg}(t)\frac{\ii}{h}\Op_h(\omega)F_h(-t) = -hF_h(t)R_h^{\rm Eg}(t)F_h(-t).
\end{align*}

Since $\mathsf D_h^{\rm Eg}(0)=0$,   
\begin{equation*}
    F_h(t)\mathsf D_h^{\rm Eg}(t)F_h(-t)=-h\int_0^tF_h(s)R_h^{\rm Eg}(s)F_h(-s)\dd s.
\end{equation*}
$R_h^{\rm Eg}(t)$ and $F_h(t)$ are uniformly bounded order-zero semiclassical pseudodifferential operators, so the integral, and hence $h^{-1}\mathsf D_h^{\rm Eg}(t)$, is a uniformly bounded order-zero semiclassical pseudodifferential operator. This proves \eqref{eq:egorov-with-class}.
\end{proof}

Now we study the leading amplitude transport on the positive branch introduced by $\beta_h$. Its principal coefficient $\beta=\gamma+\kappa$ contains both the physical material attenuation $\gamma$ and the geometric transport correction $\kappa$. According to discussions in section \ref{sec:injection-symbol} and \ref{sec:extraction-symbol}, the leading order part is described by
\begin{equation}\label{eq:damping-integral}
    \mathcal G_t(\bm{\rho}) := \int_0^t \beta(\bm{\Phi}^{-s}(\bm{\rho}))\dd s = \int_0^t\bigl(\gamma+\kappa\bigr)(\bm{\Phi}^{-s}(\bm{\rho}))\dd s.
\end{equation}

\begin{theorem}[Positive-branch damping factor]
For $0\le t\le T+\delta_T$, the functions $\mathcal G_t$ and $e^{-\mathcal G_t}$ are local symbols of order $0$ on the selected flow tube, uniformly in $t$. After the fixed microlocal cutoffs,
\begin{equation}\label{eq:damped-factorization}
    S_h(t)=\Op_h(e^{-\mathcal G_t})F_h(t)+h\mathcal R_h^{\rm damp}(t),
\end{equation}
where $\mathcal R_h^{\rm damp}(t)$ is an order-zero propagating operator along $\bm\Phi^t$, with uniform amplitude bounds. 
\end{theorem}

\begin{proof}
From \eqref{eq:damping-integral},
$\partial_t\mathcal G_t(\bm\rho)=\beta(\bm\Phi^{-t}(\bm\rho))$. Moreover,
\begin{align*}
\bm H_\omega\mathcal G_t(\bm\rho)=\frac{\dd}{\dd\varepsilon}\bigg|_{\varepsilon=0}\int_0^t\beta(\bm\Phi^{\varepsilon-s}(\bm\rho))\dd s=\beta(\bm\rho)-\beta(\bm\Phi^{-t}(\bm\rho)).
\end{align*}
Thus
\begin{equation}\label{eq:proof-1}
 (\partial_t+\bm H_\omega)e^{-\mathcal G_t}
 +\beta e^{-\mathcal G_t}=0.
\end{equation}
All derivatives of $\mathcal G_t$ and $e^{-\mathcal G_t}$ are bounded on the compact flow tube because $\beta$ is a local symbol of order $0$ and the flow is smooth for bounded time.

\eqref{eq:symbol-product} gives uniformly bounded order-zero semiclassical pseudodifferential operators $R_h^{\rm com}(t)$ and $R_h^{\rm prod}(t)$ such that
\begin{align*}
 \frac{\ii}{h}[\Op_h(\omega),\Op_h(e^{-\mathcal G_t})]&=\Op_h(\bm H_\omega e^{-\mathcal G_t})+hR_h^{\rm com}(t),\\
 \Op_h(\beta_h)\Op_h(e^{-\mathcal G_t})&=\Op_h(\beta e^{-\mathcal G_t})+hR_h^{\rm prod}(t).
\end{align*}
Using \eqref{eq:lossless-propagator} and \eqref{eq:proof-1}, it follows that
\begin{equation}\label{eq:residual-bound}
 \left(\partial_t+\frac{\ii}{h}\Op_h(\omega)+\Op_h(\beta_h)\right)
 \Op_h(e^{-\mathcal G_t})F_h(t)
 =h\mathcal E_h(t),
\end{equation}
where $\mathcal E_h(t)$ is an order-zero propagating operator along $\bm\Phi^t$, with amplitudes uniformly bounded for $t\in[0,T+\delta_T]$.

Notice that $S_h(t)$ and $\Op_h(e^{-\mathcal G_t})F_h(t)$ have the same initial value $I$ at $t=0$. Duhamel's formula \eqref{eq:duhamel} applied to \eqref{eq:residual-bound} gives
\begin{equation*}
 S_h(t)-\Op_h(e^{-\mathcal G_t})F_h(t)=-h\int_0^tS_h(t-s)\mathcal E_h(s)\dd s.
\end{equation*}
Define
$$
    \mathcal R_h^{\rm damp}(t):=-\int_0^t S_h(t-s)\mathcal E_h(s)\dd s.
$$
The uniform boundedness of $\mathcal R_h^{\rm damp}$ follows from the uniform boundedness of $S_h$ and $\mathcal E_h$. This proves the theorem.
\end{proof}

\subsection{The returned field.}

The returned field contains multiplication by $\Gamma$, which changes the positive polarization into the negative TM polarization and reverses the sign of the principal temporal frequency. 

\subsubsection{The negative branch.}
We repeat the same construction for the eigenvalue $-\omega$ in \eqref{eq:leadingorder}. Its principal right and left symbols are
\begin{equation*}
 \bm r_-(\bm x,\bm\xi):=\Gamma\bm r(\bm x,\bm\xi) = \bm r(\bm x,-\bm\xi),
 \qquad
 \bm\ell_-(\bm x,\bm\xi):=\bm r_-(\bm x,\bm\xi)^*A_0(\bm x) = \bm\ell(\bm x,\bm\xi)\Gamma.
\end{equation*}
Both symbols have order $0$. The relations
$\Gamma \mathsf M_{\rm md}(\bm x,\bm\xi)\Gamma=-\mathsf M_{\rm md}(\bm x,\bm\xi)$,
$\bm r(\bm x,-\bm\xi)=\Gamma\bm r(\bm x,\bm\xi)$, and the evenness of $\omega$ imply
$$
 \bm{\mathcal D}(\bm x,-\bm\xi)=-\Gamma\bm{\mathcal D}(\bm x,\bm\xi),
 \qquad
 \kappa(\bm x,-\bm\xi)=-\kappa(\bm x,\bm\xi).
$$
Therefore the leading order-zero scalar coefficient on the negative branch is
\begin{equation*}
 \beta_-(\bm x,\bm\xi):=\gamma(\bm x,\bm\xi)-\kappa(\bm x,\bm\xi)=\beta(\bm x,-\bm\xi).
\end{equation*}
This formula exhibits the different behavior of the material and geometric contributions under reversal of the propagating branch. The material attenuation $\gamma$ is unchanged, since in the present Maxwell--Debye model it is determined by the dissipative material coefficients and is independent of the sign of $\bm\xi$. Nevertheless, the geometric transport correction is odd,
$$
    \kappa(\bm x,-\bm\xi)=-\kappa(\bm x,\bm\xi).
$$
Consequently the leading amplitude-transport coefficients on the outgoing positive and returning negative branches are
\begin{equation*}
    \beta_+=\gamma+\kappa, \qquad \beta_-=\gamma-\kappa.
\end{equation*}
Thus the geometric corrections occur with opposite signs on the two branches, whereas the physical attenuation has the same sign on both. This distinction will be responsible for the cancellation of geometric transport, but the addition of material loss, in the round-trip returned symbol. 

Repeating the steps in section \ref{sec:injection-symbol} and \ref{sec:extraction-symbol} gives classical order-zero symbols
\begin{equation}\label{eq:full-negative-mode-symbols}
 \begin{aligned}
 \bm r_{-,h}&=\bm r_-+h\widetilde{\bm r}_{-,h},\\
 \bm\ell_{-,h}&=\bm\ell_-+h\widetilde{\bm\ell}_{-,h},\\
 \beta_{-,h}&=\beta_-+h\widetilde\beta_{-,h},
 \end{aligned}
\end{equation}
where the three correction symbols are uniformly of order $-1$. With
\begin{equation*}
 J_{h,-}:=\Op_h(\bm r_{-,h}), \qquad J_{h,-}^{\#}:=\Op_h(\bm\ell_{-,h}),
\end{equation*}
one has
\begin{align}
 \left(\partial_t+\frac{\ii}{h}\Op_h(\mathsf M_{\rm md})+\Op_h(\mathsf C_{\rm md})\right)J_{h,-}
 &=J_{h,-}\left(\partial_t-\frac{\ii}{h}\Op_h(\omega)+\Op_h(\beta_{-,h})\right)
 +R_h^{\infty,-},
 \label{eq:negative-intertwine-relation}\\
 J_{h,-}^{\#}\left(\partial_t+\frac{\ii}{h}\Op_h(\mathsf M_{\rm md})+\Op_h(\mathsf C_{\rm md})\right)
 &=\left(\partial_t-\frac{\ii}{h}\Op_h(\omega)+\Op_h(\beta_{-,h})\right)
 J_{h,-}^{\#}+\widetilde R_h^{\infty,-}.
 \label{eq:negative-left-intertwine-relation}
\end{align}
Writing $\Pi_-:=\bm r_-\bm\ell_-$, which is a local matrix symbol of order $0$, the normalization may be imposed simultaneously in the form
\begin{equation*}
 J_{h,-}^{\#}J_{h,-}=I+R_h^{\infty,{\rm n},-}, \qquad J_{h,-}J_{h,-}^{\#}=\Op_h(\Pi_-)+hR_h^{{\rm p},-},
\end{equation*}
where $R_h^{{\rm p},-}$ is a uniformly bounded order $-1$ semiclassical pseudodifferential operator and $R_h^{\infty,-}, \widetilde R_h^{\infty,-}, R_h^{\infty,{\rm n},-}$ are negligible.

Since $\Gamma\bm r=\bm r_-$, the corrected injections satisfy only the order-$h$ relation
\begin{equation*}
 \Gamma J_h-J_{h,-}=hR_h^\Gamma.
\end{equation*}
Here $R_h^\Gamma$ is a uniformly bounded order $-1$ semiclassical pseudodifferential operator from scalar to vector fields.

Let $S_{-,h}(t)$ be the strongly continuous propagator on $\mathcal H_{\rm s}$, with domain $H_h^1(\mathbb R^2)$, defined by
\begin{equation}\label{eq:negative-propagator}
\left(\partial_t-\frac\ii h\Op_h(\omega)+\Op_h(\beta_{-,h})\right)S_{-,h}(t)=0,
\qquad S_{-,h}(0)=I.
\end{equation}
It is uniformly bounded on $H_h^s$ for every fixed $s$ and for $t$ in fixed compact intervals. Duhamel's formula applied to
\eqref{eq:negative-intertwine-relation}--\eqref{eq:negative-left-intertwine-relation} gives, uniformly for $|t|\le T_0$ with fixed $T_0$,
\begin{align}
 U_0(t)J_{h,-}&=J_{h,-}S_{-,h}(t)+\mathcal R_h^{\infty,-}(t)\notag,\\
 J_{h,-}^{\#}U_0(t)&=S_{-,h}(t)J_{h,-}^{\#}+\widetilde{\mathcal R}_h^{\infty,-}(t)
 \label{eq:negative-left-propagator-intertwining}
\end{align}
with $\mathcal R_h^{\infty, -}$ and $\widetilde{\mathcal R}_h^{\infty,-}$ negligible.

Define
\begin{equation*}
 \mathcal G_{t,-}(\bm\rho):=\int_0^t\beta_-(\bm\Phi^s(\bm\rho))\dd s.
\end{equation*}
The same transport argument, now using $\partial_t-\bm H_\omega$, gives an order-zero propagating operator $\mathcal R_{-,h}^{\rm damp}(t)$ along $\bm\Phi^{-t}$, with uniform amplitude bounds, such that
\begin{equation}\label{eq:s-h-factorization}
S_{-,h}(t)=\Op_h(e^{-\mathcal G_{t,-}})F_h(-t)+h\mathcal R_{-,h}^{\rm damp}(t).
\end{equation}
For every fixed $s$ and $T_0$,
\begin{equation*}
 \sup_{|t|\le T_0}\|h\mathcal R_{-,h}^{\rm damp}(t)\|_{H_h^s(\mathbb R^2)\to H_h^s(\mathbb R^2)}\le C_{s,T_0}h.
\end{equation*}

\subsubsection{The returned operator and its principal symbol}

Now we are ready to analyze the returned operator $K_{m,\tau_h}$ defined in \eqref{eq:Kmdef} from semiclassical perspective. By reformulating \eqref{eq:physical-reversed-operator} with positive-mode scalar propagators, We define the scalar operator $\mathcal K_{m,\tau_h}$ corresponding to $K_{m,\tau_h}$. 

\begin{proposition}
Let $m\in C_c^\infty(\mathbb R^2;\mathbb R)$, $M_m$ be multiplication by $m$, and let $K_{m,\tau_h}$ be defined by \eqref{eq:Kmdef}. Define the scalar operator on $\mathcal H_{\rm s}$ by
\begin{equation*}
\mathcal K_{m,\tau_h}:=\frac12\int_{-\tau_h/2}^{\tau_h/2}\eta_{\tau_h}(s)S_{-,h}(T-s)M_mS_h(T+s)\dd s.
\end{equation*}
Microlocally on $K$, after inserting fixed compactly supported input and output cutoffs, there are a compactly supported symbol $r_{m,h}^{\rm br}$ of order $0$ and a negligible remainder $R_{m,h}^{\infty,\rm br}$ such that
\begin{equation}\label{eq:physical-scalar-bridge}
J_{h,-}^{\#}K_{m,\tau_h}J_h=\mathcal K_{m,\tau_h}+h\Op_h(r_{m,h}^{\rm br})+R_{m,h}^{\infty,\rm br}.
\end{equation}
For every $N$ there are an integer $N'=N'(N)$ and a constant $C_N$, independent of $h$ and of $m$ with support in a fixed compact set, such that
\begin{equation}\label{eq:bridge-symbol-bound}
 \|r_{m,h}^{\rm br}\|_{0,N}\le C_N\max_{|\alpha|\le N'}\|\partial_{\bm x}^\alpha m\|_{L^\infty}.
\end{equation}
\end{proposition}

\begin{proof}
\eqref{eq:intertwining-operator-remainder} and \eqref{eq:negative-left-propagator-intertwining} give, uniformly for $|s|\le\tau_h/2$,
\begin{align*}
 U_0(T+s)J_h&=J_hS_h(T+s)+E_{+,s,h}^{\infty},\\
 J_{h,-}^{\#}U_0(T-s)&=S_{-,h}(T-s)J_{h,-}^{\#}+E_{-,s,h}^{\infty},
\end{align*}
where $E_{+,s,h}^{\infty}$ and $E_{-,s,h}^{\infty}$ are negligible.

Symbolic composition, using the principal identity $\bm r_-=\Gamma\bm r$, gives
\begin{equation}\label{eq:m-mode-composition}
 J_{h,-}^{\#}M_m\Gamma J_h=M_m+h\Op_h(c_{m,h})+R_{m,h}^{\infty},
\end{equation}
where $c_{m,h}$ is a compactly supported symbol of order $-1$ and $R_{m,h}^{\infty}$ is negligible.

The leading symbol in \eqref{eq:m-mode-composition} is $\bm\ell_-m\Gamma\bm r=m$. The corrections in \eqref{eq:full-mode-symbols} and \eqref{eq:full-negative-mode-symbols}, together with \eqref{eq:symbol-product}, show directly that every remaining non-smoothing term is $h$ times a compactly supported symbol of order $-1$.

Substituting these relations into the integrand defining $K_{m,\tau_h}$ shows that its compressed integrand equals
\begin{equation*}
 S_{-,h}(T-s)M_mS_h(T+s)+hA_{m,s,h}+R_{m,s,h}^{\infty},
\end{equation*}
Here $R_{m,s,h}^{\infty}$ is negligible and $A_{m,s,h}$ is an order-zero propagating operator along $\bm\Phi^{2s}$, because
$\bm\Phi^{-(T-s)}\circ\bm\Phi^{T+s}=\bm\Phi^{2s}$. Its amplitude bounds depend on only finitely many derivatives of $m$. Since $s=\nu h\vartheta$ with $|\vartheta|\le1/2$, this is a short-time operator. After the fixed cutoffs, it can therefore be written as $\Op_h(a_{m,s,h})$ plus a negligible remainder, where $a_{m,s,h}$ is a compactly supported order-zero symbol whose bounds are uniform in $s/h$. Finally,
\begin{equation*}
 \int_{-\tau_h/2}^{\tau_h/2}|\eta_{\tau_h}(s)|\dd s=\int_{-1/2}^{1/2}|\eta(\vartheta)|\dd\vartheta,
\end{equation*}
so integration preserves every finite derivative bound stated above. This proves \eqref{eq:physical-scalar-bridge}.
\end{proof}

Notice that the group property also gives the exact time-shift identity
\begin{equation*}
\bm{w}_R^m(2T+r)=U_0(r)K_{m,\tau_h}\bm{\Upsilon}_0
\end{equation*}
On the selected negative branch, an additional time $r=h\zeta$ has principal phase $e^{\ii\zeta\omega}$. This result will be needed when proving refocusing occurrence. 

Now we calculate the principal symbol of the returned operator. With the Fourier-transform convention
$$
    \widehat\eta(\varrho):=\int_{\mathbb R}e^{-\ii\varrho s}\eta(s)\dd s,
$$
we define the pulse response
\begin{equation}\label{eq:pulse-response}
p_\eta(\varpi):=\frac12\widehat\eta(-2\nu\varpi).
\end{equation}
For a phase point $\bm\rho$, define the spatial point reached by the outgoing positive packet after time $t$ by
\begin{equation*}
 \bm{X}_t(\bm{\rho}):=\bm{\pi}_{\bm{x}}\bm{\Phi}^t(\bm{\rho}).
\end{equation*}
The positive one-way damping written at the source phase point is
\begin{equation*}
 \mathcal G_t^{\rm out}(\bm{\rho}):=\int_0^t\beta(\bm{\Phi}^s(\bm{\rho}))\dd s =\mathcal G_t(\bm{\Phi}^t(\bm{\rho})).
\end{equation*}
\begin{theorem}\label{thm:symbol-tailored}
Let $q\in C_c^\infty(\mathcal U)$ with $\supp q\subset K$, put $Q_h=\Op_h(q)$, and let $m\in C_c^\infty(\mathbb R^2;\mathbb R)$. Then
\begin{equation}\label{eq:returned-symbol-expansion}
 Q_h\mathcal K_{m,\tau_h}Q_h
 =\Op_h(q^2j_m^+)+h\Op_h(r_{m,h})+R_{m,h}^{\infty},
\end{equation}
where $r_{m,h}$ is a compactly supported symbol of order $0$, the remainder $R_{m,h}^{\infty}$ is negligible from $H_h^{-N}(\mathbb R^2)$ to $H_h^N(\mathbb R^2)$ for every $N$, and
\begin{equation}\label{eq:returned-principal-symbol}
 j_m^+(\bm\rho)=\overline{p_\eta(\omega(\bm\rho))}m(\bm X_T(\bm\rho))\exp\left[-\mathcal G_T^{\rm out}(\bm\rho)-\mathcal G_{T,-}(\bm\rho)\right].
\end{equation}
For every $N$ there are $N'=N'(N)$ and $C_N$ such that
\begin{equation}\label{eq:returned-symbol-seminorm-bound}
 \|r_{m,h}\|_{0,N}\le C_N\max_{|\alpha|\le N'}\|\partial_{\bm x}^\alpha m\|_{L^\infty}.
\end{equation}

The total principal round-trip damping is
\begin{equation}\label{eq:round-trip-damping}
 \mathcal G_T^{\rm out}(\bm\rho)+\mathcal G_{T,-}(\bm\rho)=2\int_0^T\gamma(\bm\Phi^s(\bm\rho))\dd s.
\end{equation}
\end{theorem}

\begin{proof}
All operators are first localized to one fixed compact flow tube containing the trajectories issued from $\supp q$. Since $\tau_h=\nu h$, for sufficiently small $h$ we have $|s|\leq \tau_h/2 < \delta_T$. Hence $T\pm s$ remain in the fixed time interval covered by $K_{T,\delta_T}$, and all propagator and symbol estimates below are uniform for $|s|\leq \tau_h/2$. The positive and negative factorizations \eqref{eq:damped-factorization} and \eqref{eq:s-h-factorization} give
\begin{align}
Q_h\mathcal K_{m,\tau_h}Q_h=&\frac12\int_{-\tau_h/2}^{\tau_h/2}\eta_{\tau_h}(s)Q_h\Op_h(e^{-\mathcal G_{T-s,-}})F_h(-T+s)M_m\Op_h(e^{-\mathcal G_{T+s}})F_h(T+s)Q_h\dd s\notag\\
&\quad+h\Op_h(r_{m,h}^{(1)})+R_{m,h}^{\infty,(1)},
\label{eq:returned-main-integral}
\end{align}
where $r_{m,h}^{(1)}$ is a compactly supported order-zero symbol satisfying \eqref{eq:returned-symbol-seminorm-bound} and $R_{m,h}^{\infty,(1)}$ negligible.

Set $s=\nu h\vartheta$, $|\vartheta|\le1/2$. Then $\eta_{\tau_h}(s)\dd s=\eta(\vartheta)\dd\vartheta$. The group property gives
\begin{equation*}
F_h(-T+s)M_m\Op_h(e^{-\mathcal G_{T+s}})F_h(T+s)=F_h(s)\Bigl[F_h(-T)\Op_h(me^{-\mathcal G_{T+s}})F_h(T)\Bigr]F_h(s).
\end{equation*}
Applying Theorem \ref{thm:egorov} to the family $me^{-\mathcal G_{T+s}}$. Since $|s|\leq\tau_h/2=\nu h/2$, this family has uniform order-zero symbol bounds. Hence
\begin{equation}\label{eq:returned-egorov-step}
 F_h(-T)\Op_h(me^{-\mathcal G_{T+s}})F_h(T) = \Op_h\left((me^{-\mathcal G_{T+s}})\circ\bm\Phi^T\right)+h\Op_h(r_{m,s,h}^{\rm Eg})+R_{m,s,h}^{\infty,\rm Eg},
\end{equation}
where $R_{m,s,h}^{\infty,\rm Eg}$ is negligible, and $r_{m,s,h}^{\rm Eg}$ is a compactly supported order-zero symbol, uniformly in $s/h$, and satisfies the bounds in \eqref{eq:returned-symbol-seminorm-bound}. 
Since
\begin{equation*}
 \mathcal G_{T+s}(\bm\Phi^T(\bm\rho))=\int_{-s}^{T}\beta(\bm\Phi^u(\bm\rho))\dd u,
\end{equation*}
Taylor's formula, together with the uniform derivative bounds on the compact flow tube, gives
\begin{equation*}
 e^{-\mathcal G_{T+s}\circ\bm\Phi^T}=e^{-\mathcal G_T^{\rm out}}+h g_{+,\vartheta,h}.
\end{equation*}
Here $g_{+,\vartheta,h}$ is a compactly supported order-zero symbol, uniformly for $|\vartheta|\le1/2$. Hence the principal symbol in \eqref{eq:returned-egorov-step} is
$m(\bm X_T)e^{-\mathcal G_T^{\rm out}}$.

$F_h(\nu h\vartheta)$ satisfies
$$
    \partial_\vartheta F_h(\nu h\vartheta) +\ii\nu\Op_h(\omega)F_h(\nu h\vartheta)=0, \qquad F_h(0)=I. 
$$
Solving the corresponding symbolic evolution gives
$$
    F_h(\nu h\vartheta) = \Op_h(e^{-\ii\nu\vartheta\omega}+hf_{\vartheta,h})+R_{\vartheta,h}^{\infty,F},
$$
where $f_{\vartheta,h}$ is uniformly of order zero for $|\vartheta|\leq 1/2$, and $R_{\vartheta,h}^{\infty, F}$ is negligible.

Likewise,
\begin{equation}\label{eq:negative-damping-symbol-expansion}
 e^{-\mathcal G_{T-s,-}}=e^{-\mathcal G_{T,-}}+h g_{-,\vartheta,h},
\end{equation}
where $g_{-,\vartheta,h}$ is a compactly supported order-zero symbol, uniformly for $|\vartheta|\le1/2$. Combining \eqref{eq:returned-main-integral}--\eqref{eq:negative-damping-symbol-expansion} with \eqref{eq:symbol-product} gives, uniformly in $\vartheta$,
\begin{equation*}
 Q_h\Op_h\left(e^{-2\ii\nu\vartheta\omega}m(\bm X_T)e^{-\mathcal G_T^{\rm out}-\mathcal G_{T,-}}\right)Q_h
 =\Op_h\left(q^2e^{-2\ii\nu\vartheta\omega}
 m(\bm X_T)e^{-\mathcal G_T^{\rm out}-\mathcal G_{T,-}}
 +h r_{m,\vartheta,h}^{(2)}\right)+R_{m,\vartheta,h}^{\infty,(2)},
\end{equation*}
where $r_{m,\vartheta,h}^{(2)}$ is a compactly supported order-zero symbol with the stated uniform derivative bounds and $R_{m,\vartheta,h}^{\infty,(2)}$ is negligible.

Integration in $\vartheta$ preserves each finite derivative bound. Using
\begin{equation*}
 \frac12\int_{-1/2}^{1/2}\eta(\vartheta)e^{-2\ii\nu\vartheta\omega}\dd\vartheta=\frac12\widehat\eta(2\nu\omega)=\overline{p_\eta(\omega)},
\end{equation*}
we obtain \eqref{eq:returned-symbol-expansion}, \eqref{eq:returned-principal-symbol}, and \eqref{eq:returned-symbol-seminorm-bound}. Since $\beta=\gamma+\kappa$ and $\beta_-=\gamma-\kappa$, identity \eqref{eq:round-trip-damping} follows immediately.
\end{proof}

\begin{remark}
The decomposition
$$
    \beta_+=\gamma+\kappa, \qquad \beta_-=\gamma-\kappa
$$
makes the structure of the returned amplitude transparent. The coefficient $\gamma$ is the material attenuation rate obtained from the dissipative matrix $B_0$, whereas $\kappa$ is the geometric transport correction generated by the phase-space variation of the selected polarization. Hence, along the matched outgoing and returning rays,
$$
\begin{aligned}
    \mathcal G_T^{\rm out}(\bm\rho)+\mathcal G_{T,-}(\bm\rho)&=\int_0^T\left[(\gamma+\kappa)+(\gamma-\kappa)\right](\bm\Phi^s(\bm\rho))\dd s\\
    &=2\int_0^T\gamma(\bm\Phi^s(\bm\rho))\dd s.
\end{aligned}
$$
Thus the geometric transport corrections cancel between the two branches, while the material attenuation is accumulated twice, once during outgoing propagation and once during the return. Accordingly, the principal returned symbol contains the physical round-trip attenuation factor
$$
    \exp\left[-2\int_0^T\gamma(\bm\Phi^s(\bm\rho))\dd s\right].
$$
\end{remark}

In subsequent discussions, we require the modulation function $m$ and mirror time $T$ varying with $h$. Notice that all constructions above are uniform when $T$ ranges in a fixed compact subset of $(0,\infty)$. Each finite symbol-derivative bound uses only finitely many derivatives of $m$. Likewise, each Sobolev remainder estimate uses only finitely many derivatives of $m$. Therefore, theorem \ref{thm:symbol-tailored} could be generalized to the following proposition.

\begin{proposition}
Let $T=T_h$ depending on $h$ stay in a compact subset of $(0,\infty)$ and let $m=m_h$ be real, $h$--dependent and supported in one fixed compact set, satisfy
\begin{equation*}
1+\eta_{\tau_h}(t-T_h)m_h(\bm x)\ge\kappa_0,
\end{equation*}
and satisfy
\begin{equation*}
 \sup_{0<h\le h_0}\max_{|\alpha|\le N}\|\partial_{\bm x}^{\alpha}m_h\|_{L^\infty}<\infty \qquad \text{for every }N\in\mathbb N_0.
\end{equation*}
Then
\begin{equation}\label{eq:uniform-m-symbol-remainder}
 Q_h\mathcal K_{m_h,\tau_h}Q_h = \Op_h(q^2j_{m_h}^+)+h\Op_h(r_{m_h,h})+R_{m_h,h}^{\infty},
\end{equation}
where $R_{m_h,h}^\infty$ is negligible, and $r_{m_h,h}$ is a compactly supported order-zero symbol such that for every $N\in\mathbb{N}_0$,
\begin{equation*}
 \|r_{m_h,h}\|_{0,N} \le C_N\max_{|\alpha|\le N'(N)}\|\partial_{\bm x}^\alpha m_h\|_{L^\infty}.
\end{equation*}
\end{proposition}

\section{Design of modulation function}\label{sec:necessity-tailored-annulus}

The preceding section characterizes the principal returned-wave symbol generated by a prescribed spatial modulation. We now return to the forced Maxwell--Debye system generated by the point source and address the constructive problem of choosing an admissible, $h$-dependent modulation $m_h$. The construction is causal in the sense that $m_h$ is determined from measurements of the outgoing field available before the material pulse is applied.

In this section we first identify the leading positive-TM component generated by the short point-source pulse. Then we isolate a regular visible component of the source-side phase space, whose outgoing trajectories reach the annular control region, and parameterize this component by frequency and direction. Third, we apply the known lossless backward propagator to the measured outgoing field, form a normalized directional energy, and define $m_h$ reciprocally along the corresponding pulse-center incidence curve.

The modulation and measurements are localized to a bounded annular control region $\mathcal A\subset\mathbb R^2$ surrounding a known bounded set containing the inhomogeneity. After the fixed annular input and output cutoffs, let
$$
\mathcal O_{\mathcal A}:\mathcal H_{\rm v}\to\mathcal H_{\rm s}
$$
be a uniformly bounded order-zero semiclassical pseudodifferential observation operator. Since $\Gamma$ is known and acts only on the field components, we incorporate it into the positive-branch observation operator and set
$$
\mathcal O_{\mathcal A}^+:=\mathcal O_{\mathcal A}\Gamma.
$$
Let $o_+$ denote its scalar principal symbol on the outgoing positive TM branch. Then $o_+$ is a compactly supported order-zero symbol on $\mathcal U$, and, microlocally in the annulus,
\begin{equation}\label{eq:observation-symbol-expansion}
\mathcal O_{\mathcal A}^+J_h = \Op_h(o_+)+hR_h^{\rm obs},
\end{equation}
where $R_h^{\rm obs}$ is a uniformly bounded order-zero semiclassical pseudodifferential operator on $\mathcal H_{\rm s}$.

Throughout Sections~\ref{sec:necessity-tailored-annulus}--\ref{sec:causal-refocusing}, we assume that the material-pulse profile $\eta$ is real and even. Consequently, the pulse-response factor $p_\eta$ defined in \eqref{eq:pulse-response} is real. We restrict each retained frequency interval to a region on which $p_\eta$ has a fixed nonzero sign and, without loss of generality, take this sign to be positive.

\subsection{Point source and its selected positive-mode amplitude}

Fix a compact set $K\subset\mathcal U$ whose flow tube for the mirror times used below remains in $\mathcal U$. Recall that in the $4\times4$ TM system the corresponding background field $\bm{W}_h$ satisfies
\begin{equation*}
    \partial_t\bm{W}_h+A_0^{-1}\Lambda\bm{W}_h+A_0^{-1}B_0\bm{W}_h = -A_0^{-1}g_h(t)\delta(\bm{x}-\bm{z}_\star)\bm{e}_3, \qquad \bm{e}_3=(0,0,1,0)^\top.
\end{equation*}

The positive TM eigenvector \eqref{eq:explicit-r-positive-branch} gives the source-to-mode coupling
$$
    q_\star:=-\bm{r}(\bm{z}_\star,\bm{\xi})^*\bm{e}_3=-\frac{1}{\sqrt{2\epsilon_0\epsilon_\infty(\bm{z}_\star)}},
$$
which is nonzero and independent of $\bm{\xi}$. We therefore define the selected source spectrum by
\begin{equation}\label{eq:source-spectrum}
    s_h(\varpi):= q_\star\hat{g}_h\left(\frac{\varpi}{h}\right) = q_\star pe^{-\ii \zeta_{\rm src}\varpi}\hat{f}(\varpi).
\end{equation}

Let $\chi_\star\in C_c^\infty(\mathbb R_{\bm{\xi}}^2\setminus\{0\};[0,1])$ which satisfies
\begin{equation}\label{eq:source-cutoff}
    \{\bm{z}_\star\}\times\supp\chi_\star\subset K\subset\mathcal U.
\end{equation}
Define $Q_{\star,h}:=\Op_h(\chi_\star)$. In the subsequent discussions, we will repeatedly apply truncation $Q_{\star,h}$ to ensure everything is within the exact tube $K$. 

As stated in Proposition~\ref{prop:backpropagating}, the back-propagating wave is determined by applying $K_{m,\tau_h}$ to $\bm{\Upsilon}_{0,h}:=\partial_t\bm{W}_h(0)$. Now we compress $\bm{\Upsilon}_{0,h}$ to the positive branch. 

\begin{lemma}\label{lem:short-time-point-parametrix}
Let $\bm z\in\mathbb R^2$ and $\chi_0\in C_c^\infty(\mathbb R_{\bm\xi}^2\setminus\{0\})$. For every fixed $R>0$ and $|\zeta|\le R$,
\begin{equation}\label{eq:short-time-point-parametrix}
 \Op_h(\chi_0)S_h(-h\zeta)\delta_{\bm z}
 =\frac1{(2\pi h)^2}\int_{\mathbb R^2}e^{\ii(\bm x-\bm z)\cdot\bm\xi/h}
 \left[e^{\ii\zeta\omega(\bm z,\bm\xi)}\chi_0(\bm\xi)
 +h b_{\zeta,h}(\bm x,\bm\xi)\right]\dd\bm\xi,
\end{equation}
where $b_{\zeta,h}$ is a controlled amplitude, uniformly for $|\zeta|\le R$. Thus its $\bm\xi$-support lies in one fixed compact subset of $\mathbb R^2\setminus\{0\}$, and all its derivatives are uniformly bounded on compact $\bm x$-sets.
\end{lemma}

\begin{proof}
Since $S_h$ is only valid on a compact frequency region where the positive branch is controlled, we have to choose truncation function $\widetilde\chi_0\in C_c^\infty(\mathbb R^2_{\bm\xi}\setminus\{0\})$, which is equal to one on a neighborhood of $\supp\chi_0$, and apply $S_h$ onto $\widetilde\chi_0\delta_{\bm{z}}$. 

Differentiating $S_h(-h\zeta)$ gives
\begin{equation*}
 \partial_\zeta S_h(-h\zeta)=\ii\Op_h(\omega)S_h(-h\zeta)+h\Op_h(\beta_h)S_h(-h\zeta),
\end{equation*}
and the leading symbol equation is $\partial_\zeta s_0=\ii\omega s_0$, $s_0|_{\zeta=0}=1$, hence $s_0=e^{\ii\zeta\omega}$. The successive transport equations preserve the stated derivative and support bounds uniformly on bounded $\zeta$-intervals. So solving the symbol equation term by term gives
\begin{equation}\label{eq:shorttime-expansion}
 \Op_h(\chi_0)S_h(-h\zeta)\Op_h(\widetilde\chi_0)
 =\Op_h\left(\chi_0e^{\ii\zeta\omega}+h b_{\zeta,h}^{(1)}+h^2r_{\zeta,h}^{(2)}\right)
 +R_{\zeta,h}^{\infty},
\end{equation}
where $b_{\zeta,h}^{(1)}$ and $r_{\zeta,h}^{(2)}$ are controlled amplitudes, uniformly for $|\zeta|\le R$, and $R_{\zeta,h}^{\infty}$ is negligible.

Fix $N_0>1$. In two-dimensional case,
\begin{equation}\label{eq:delta-semiclassical-sobolev-bound}
 \delta_{\bm z}\in H_h^{-N_0}(\mathbb R^2),
 \qquad
 \|\delta_{\bm z}\|_{H_h^{-N_0}(\mathbb R^2)}\le C_{N_0}h^{-1},
\end{equation}
uniformly in $\bm z$. Hence any family which is smoothing to all orders sends $\delta_{\bm z}$ to $O(h^L)$ in $H_h^K$ for arbitrary prescribed $K,L$, after choosing the input/output Sobolev index and re-indexing the arbitrary power in the smoothing estimate. The insertion of $\Op_h(\widetilde\chi_0)$ on the right changes the left-hand side of \eqref{eq:short-time-point-parametrix}, after the output cutoff $\Op_h(\chi_0)$, by such a scalar smoothing family. Its Schwartz kernel applied to $\delta_{\bm z}$ is therefore absorbed into the remainder. Lastly, applying \eqref{eq:shorttime-expansion} to $\delta_{\bm z}$ gives an oscillatory integral with amplitude
$e^{\ii\zeta\omega(\bm x,\bm\xi)}\chi_0(\bm\xi)$ plus $h$ times a controlled amplitude. Taylor's formula writes
\begin{equation*}
 e^{\ii\zeta\omega(\bm x,\bm\xi)}-e^{\ii\zeta\omega(\bm z,\bm\xi)}=\sum_{j=1}^2(x_j-z_j)c_{j,\zeta}(\bm x,\bm z,\bm\xi),
\end{equation*}
with $c_{j,\zeta}$ satisfying the same derivative and support bounds. Since
\begin{equation}\label{eq:trick}
 (x_j-z_j)e^{\ii(\bm x-\bm z)\cdot\bm\xi/h}=\frac h\ii\partial_{\xi_j}e^{\ii(\bm x-\bm z)\cdot\bm\xi/h},
\end{equation}
integration by parts in $\bm\xi$ places the frozen-symbol difference, the $h^2r_{\zeta,h}^{(2)}$ term, and the smoothing kernel in the single amplitude $h b_{\zeta,h}$ of \eqref{eq:short-time-point-parametrix}.
\end{proof}

\begin{proposition}
Microlocally on the compact set in \eqref{eq:source-cutoff},
\begin{equation}\label{eq:source-state}
    Q_{\star,h}J_h^{\#}\bm{\Upsilon}_{0,h}(\bm x)
    =\frac1{(2\pi h)^2}\int_{\mathbb R^2}e^{\ii(\bm x-\bm z_\star)\cdot\bm\xi/h}\bigg[-\frac{\ii}{h}\omega(\bm z_\star,\bm\xi)
    s_h\bigl(\omega(\bm z_\star,\bm\xi)\bigr)\chi_\star(\bm\xi)+r_{\Upsilon,h}(\bm x,\bm\xi)\bigg]\dd\bm\xi,
\end{equation}
where $r_{\Upsilon,h}$ is a controlled amplitude. More explicitly, its $\bm\xi$-support is contained in one fixed compact subset of $\mathbb R^2\setminus\{0\}$, and for every $\chi\in C_c^\infty(\mathbb R^2_{\bm x})$ and every $N$,
\begin{equation}\label{eq:source-remainder-seminorm}
 \max_{|\alpha|+|\beta|\le N}\sup_{\bm x,\bm\xi}\bigl|\partial_{\bm x}^{\alpha}\partial_{\bm\xi}^{\beta}(\chi(\bm x)r_{\Upsilon,h}(\bm x,\bm\xi))\bigr|\le C_{\chi,N}
\end{equation}
uniformly for small $h$.
\end{proposition}

\begin{proof}
Fix $N_0>1$. By \eqref{eq:delta-semiclassical-sobolev-bound}, $A_0^{-1}\bm e_3\delta_{\bm z_\star}\in H_h^{-N_0}(\mathbb R^2;\mathbb C^4)$
with norm $O(h^{-1})$. The fixed-time vector and scalar propagators and all order-zero microlocal cutoffs are bounded on the corresponding $H_h^s$ spaces, so Duhamel's formula below is well defined in $H_h^{-N_0}$.

For $-2\zeta_{\rm src}h\le s\le0$, write $\vartheta=s/h$. Symbolic composition of $Q_{\star,h}J_h^{\#}$ with the vector source injection, followed by the short-time scalar construction in Lemma~\ref{lem:short-time-point-parametrix}, gives
\begin{equation}\label{eq:source-short-time-kernel}
 Q_{\star,h}J_h^{\#}U_0(-s)A_0^{-1}\bm e_3\delta_{\bm z_\star}
 =\frac1{(2\pi h)^2}\int e^{\ii(\bm x-\bm z_\star)\cdot\bm\xi/h}
 \bigl[e^{\ii s\omega(\bm z_\star,\bm\xi)/h}c_\star(\bm\xi)\chi_\star(\bm\xi)+h a_{\vartheta,h}(\bm x,\bm\xi)\bigr]\dd\bm\xi,
\end{equation}
where $c_\star(\bm\xi)=\bm r(\bm z_\star,\bm\xi)^*\bm e_3$ is the principal scalar source coupling and $a_{\vartheta,h}$ is a controlled amplitude, uniformly for $-2\zeta_{\rm src}\le\vartheta\le0$. To obtain \eqref{eq:source-short-time-kernel}, the $\bm x$-dependence of the principal symbols is frozen at $\bm z_\star$ by Taylor's formula and the identity \eqref{eq:trick}. Every frozen-symbol difference contributes one factor $h$ and preserves the stated support and derivative bounds.

Duhamel's formula in $H_h^{-N_0}(\mathbb R^2;\mathbb C^4)$ gives
\begin{equation*}
 \bm W_h(0)=-\int_{-2\zeta_{\rm src}h}^{0}U_0(-s)A_0^{-1}\bm e_3\delta_{\bm z_\star}g_h(s)\dd s.
\end{equation*}
$g_h$ has uniformly bounded $L^1$ norm, so integration of the $h a_{\vartheta,h}$ term in \eqref{eq:source-short-time-kernel} remains $h$ times a controlled amplitude. With \eqref{eq:source-spectrum}, the leading time integral is exactly $s_h(\omega(\bm z_\star,\bm\xi))$. Hence
\begin{equation}\label{eq:source-state-before-time-derivative}
 Q_{\star,h}J_h^{\#}\bm W_h(0)
 =\frac1{(2\pi h)^2}\int e^{\ii(\bm x-\bm z_\star)\cdot\bm\xi/h}
 \left[s_h\bigl(\omega(\bm z_\star,\bm\xi)\bigr)\chi_\star(\bm\xi)
 +h a_{0,h}(\bm x,\bm\xi)\right]\dd\bm\xi,
\end{equation}
where $a_{0,h}$ is a controlled amplitude.

The source vanishes at $t=0$. Thus the homogeneous equation $L_0\bm{W}_h=0$ and \eqref{eq:left-intertwine-relation} imply, microlocally on the selected set,
\begin{equation*}
 Q_{\star,h}J_h^{\#}\bm{\Upsilon}_{0,h}
 =-Q_{\star,h}\left(\frac{\ii}{h}\Op_h(\omega)+\Op_h(\beta_h)\right)
 J_h^{\#}\bm W_h(0)+R_h^\infty\bm W_h(0),
\end{equation*}
where $R_h^\infty$ is negligible. Applying \eqref{eq:symbol-product} to \eqref{eq:source-state-before-time-derivative}, the action of $-(\ii/h)\Op_h(\omega)$ on its leading amplitude gives
\begin{equation*}
 -\frac{\ii}{h}\omega(\bm z_\star,\bm\xi)s_h\bigl(\omega(\bm z_\star,\bm\xi)\bigr)\chi_\star(\bm\xi).
\end{equation*}
Freezing $\omega(\bm x,\bm\xi)$ at $\bm z_\star$ gains a factor $h$ by \eqref{eq:trick}, which cancels the prefactor $h^{-1}$ and leaves a controlled amplitude. The term $h a_{0,h}$, the order-zero operator $\Op_h(\beta_h)$, the composition with $Q_{\star,h}$, and the smoothing remainder all produce controlled amplitudes satisfying \eqref{eq:source-remainder-seminorm}. Collecting them gives \eqref{eq:source-state}.
\end{proof}

\subsection{Visible source-side component}

Since $f$ is nonzero and odd, its Fourier transform is purely imaginary and odd. Define 
$$
    b(\varpi):=-\ii\hat{f}(\varpi).
$$
$b\in C^\infty(\mathbb{R};\mathbb{R})$ is also odd. Since $f\not\equiv 0$, $b$ is not identically zero. After replacing $f$ by $-f$ if necessary, we can always choose a compact nondegenerate interval
\begin{equation*}
    I_{\rm src}\subset\{\varpi>0:b(\varpi)>0\}.
\end{equation*}
On this interval there exists constants $0<c_{\rm src}\leq C_{\rm src}<\infty$ such that 
$$
    c_{\rm src}\leq |s_h(\varpi)|\leq C_{\rm src}, \qquad \varpi\in I_{\rm src}.
$$
Thus $I_{\rm src}$ is a positive-frequency band on which the selected source spectrum is nonvanishing. We restrict the subsequent visibility analysis to source phase points with frequencies in this band.

We now restrict attention to the portion of the selected source spectrum that can interact with the annular control region. Having fixed a positive-frequency interval $I_{\rm src}$ on which the source spectrum does not vanish, we identify the corresponding source-side phase points whose positive-branch trajectories reach the annulus at the candidate mirror time. This visible component will provide the phase-space region on which the outgoing field is measured and the modulation is subsequently constructed.

Given a compact band $B_\star$ of source satisfying
$$
    B_\star\Subset\{(\bm{z}_\star,\bm{\xi}):\chi_\star(\bm{\xi})=1, \omega(\bm{z}_\star,\bm{\xi})\in I_{\rm src}\}\cap K.
$$
Then for a candidate mirror time $T$, only the visible set
$$
    B_{\star,T}:=\{\bm{\rho}\in B_\star:\bm X_T(\bm{\rho})\in\mathcal A\}
$$
can contribute to the annular instantaneous time mirror. We work on one compact frequency--direction chart
\begin{equation}\label{eq:frequency-chart}
    \bm\Psi_T:I_{\star,T}\times\Theta_T\longrightarrow B_{\star,T}, \qquad \omega(\bm\Psi_T(\varpi,\theta))=\varpi.
\end{equation}

Recall that the time duration of modulation pulse is $\tau_h=\nu h$. To avoid overlap between modulation pulse and source pulse, we require $T>\tau_h/2$. Define the measured image at modulation pulse onset by
$$
    \widetilde{\bm\Psi}_{T,h}(\varpi,\theta):=\bm\Phi^{T-\tau_h/2}\bigl(\bm\Psi_T(\varpi,\theta)\bigr).
$$

\begin{definition}\label{def:visible-chart}
Fix a compact interval of mirror times $\mathcal T\subset(0,\infty)$. A family of charts \eqref{eq:frequency-chart}, $T\in\mathcal T$, is called \emph{regular and visible} if, after restriction to one connected compact subpatch, the following properties hold uniformly in $T$ and small $h$. \begin{enumerate}[label=\textnormal{(\roman*)}]
\item $I_{\star,T}$ and $\Theta_T$ are compact nondegenerate intervals. Their lengths are bounded below uniformly, and $p_\eta$ has the fixed sign on $I_{\star,T}$. Without loss of generality, we assume $p_\eta$ is positive.
\item The measured sheet $\widetilde{\bm\Psi}_{T,h}(I_{\star,T}\times\Theta_T)$ is separated in phase space from every other observed sheet in the selected search region by a distance at least $d_{\rm sep}>0$.
\item The selected outgoing mode is observed elliptically:
\begin{equation}\label{eq:observation-ellipticity}
    |o_+(\widetilde{\bm\Psi}_{T,h}(\varpi,\theta))|\geq c_{\rm obs}>0
\end{equation}
on $I_{\star,T}\times\Theta_T$.
\end{enumerate}
\end{definition}

Define the pulse center point $\bm{x}_T(\theta):=\bm X_T\bigl(\bm\Psi_T(\varpi,\theta)\bigr)$. The next proposition ensures the well-definedness of $\bm{x}_T(\theta)$.
\begin{proposition}
$\bm{x}_T$ is independent of $\varpi$.
\end{proposition}
\begin{proof}
    From \eqref{eq:positive-dispersion} we know that for $\lambda>0$, 
    $$
        \partial_{\bm{\xi}}\omega(\bm{x},\lambda\bm{\xi}) = \partial_{\bm{\xi}}\omega(\bm{x},\bm{\xi}),\qquad \partial_{\bm{x}}\omega(\bm{x},\lambda\bm{\xi}) = \lambda\partial_{\bm{x}}\omega(\bm{x},\bm{\xi}).
    $$
    Suppose $(\bm{x}(t),\bm{\xi}(t))$ is the Hamiltonian flow starting from $(\bm{z}_\star,\bm{\xi})$, then 
    \begin{align*}
        \frac{\dd}{\dd t}\bm{x}(t) &= \partial_{\bm{\xi}}\omega(\bm{x}(t),\bm{\xi}(t)) = \partial_{\bm{\xi}}\omega(\bm{x}(t),\lambda\bm{\xi}(t)), \\
        \frac{\dd}{\dd t}(\lambda\bm{\xi}(t)) &= -\lambda\partial_{\bm{x}}\omega(\bm{x}(t),\bm{\xi}(t)) = -\partial_{\bm{x}}\omega(\bm{x}(t),\lambda\bm{\xi}(t)).
    \end{align*}
    This implies that $(\bm{x}(t),\lambda\bm{\xi}(t))$ is the Hamiltonian flow starting from $(\bm{z}_\star,\lambda\bm{\xi})$, and hence $\bm X_T(\bm{z}_\star,\lambda\bm{\xi}) = \bm X_T(\bm{z}_\star,\bm{\xi})$.
    
    Now consider two points $\bm\Psi_T(\varpi_1,\theta)$ and $\bm\Psi_T(\varpi_2,\theta)$. Since they share the same $\theta$, they must fall into the same ray in $\bm{\xi}$-space. Hence there must exist $\lambda$ such that $\bm\Psi_T(\varpi_1,\theta) = (\bm{z}_\star,\bm{\xi})$ while $\bm\Psi_T(\varpi_2,\theta)=(\bm{z}_\star,\lambda\bm{\xi})$. We immediately conclude that $\bm X_T(\bm\Psi_T(\varpi_1,\theta)) = \bm X_T(\bm\Psi_T(\varpi_2,\theta))$.
\end{proof}

\subsection{Determination of modulation}

The purpose of this subsection is to construct an admissible nonnegative spatial modulation from the outgoing field measured before the material pulse. The construction equalizes the resulting normalized directional energy on the pulse-center incidence curve.

Define the measured field and its causal history by
\begin{equation*}
    d_h(t):=\mathcal O_{\mathcal A}^+\bm{W}_h(t), \qquad \mathcal D_h(t):=\{d_h(s):0\le s\le t\}.
\end{equation*}
Take pulse-onset time $t_{0,h}$ in a fixed interval $0<t_-\le t_{0,h}\le t_+<\infty$ and set
\begin{equation*}
    T_h:=t_{0,h}+\frac{\tau_h}{2}
\end{equation*}
as the moment of pulse center. Choose the compact interval $\mathcal T$ in Definition~\ref{def:visible-chart} so that $T_h\in\mathcal T$ for all small $h$. We will construct the modulation only with data on $\mathcal D_h(t_{0,h})$.

For $T=T_h$, let
\begin{equation*}
    \bm{\rho}_h(\varpi,\theta):=\bm\Psi_{T_h}(\varpi,\theta)
\end{equation*}
and define the measured chart at pulse onset by
\begin{equation*}
    \bm\lambda_h(\varpi,\theta):=\widetilde{\bm\Psi}_{T_h,h}(\varpi,\theta)=\bm\Phi^{t_{0,h}}\bigl(\bm{\rho}_h(\varpi,\theta)\bigr).
\end{equation*}
Since
$$
    \frac{\dd}{\dd t}\omega(\bm{\Phi}^t(\bm{\rho})) = \{\omega,\omega\}(\bm{\Phi}^t(\bm{\rho})) = 0
$$
where $\{\cdot,\cdot\}$ is Poisson bracket, $\omega$ remains constant along the flow, so $\omega(\widetilde{\bm\Psi}_{T_h,h}(\varpi,\theta))=\varpi$. The packet travels for another time $\tau_h/2$ from pulse onset to pulse center. Define
\begin{equation*}
    \bm{x}_h^c(\theta):=\bm X_{T_h}\bigl(\bm\Psi_{T_h}(\varpi,\theta)\bigr)=\bm\pi_{\bm{x}}\bm\Phi^{\tau_h/2}\bigl(\widetilde{\bm\Psi}_{T_h,h}(\varpi,\theta)\bigr).
\end{equation*}
By Definition~\ref{def:visible-chart}, this point is independent of $\varpi$, the map $\bm{x}_h^c:\Theta_{T_h}\to\operatorname{int}\mathcal A$ is a smooth embedding, and the family has a common tubular neighborhood and uniform derivative bounds.

\begin{proposition}\label{prop:inverse-oa}
On a neighborhood of the selected measured component there exists a uniformly bounded order-zero semiclassical pseudodifferential operator $\mathcal L_h^+$ such that
\begin{equation}\label{eq:observation-inverse}
    \mathcal L_h^+\mathcal O_{\mathcal A}^+J_h
    = I+hR_h^{\rm inv},
\end{equation}
microlocally on the selected component. The operator $R_h^{\rm inv}$ is uniformly bounded on $\mathcal H_{\rm s}$ with
\begin{equation*}
 \|hR_h^{\rm inv}\|_{\mathcal H_{\rm s}\to\mathcal H_{\rm s}}\le Ch.
\end{equation*}
The full symbol of $\mathcal L_h^+$ may be chosen compactly supported with uniform order-zero derivative bounds, and its principal symbol is $o_+^{-1}$ on the selected component.
\end{proposition}

\begin{proof}
By \eqref{eq:observation-ellipticity}, $o_+$ is bounded away from zero on the selected measured component. Since this component is compact and is separated from the other observed sheets by Definition~\ref{def:visible-chart}, there exists a compactly supported cutoff $\chi$ that is equal to one on a neighborhood of it, whose support is contained in a region where $o_+\neq0$, and whose support does not meet the other observed sheets.

Define
$$
    \mathcal L_h^+:=\Op_h\left(\frac{\chi}{o_+}\right).
$$
Its symbol is of order zero, has uniform derivative bounds, and agrees with $o_+^{-1}$ on the selected component.

Using \eqref{eq:observation-symbol-expansion},
$$
    \mathcal O_{\mathcal A}^+J_h =\Op_h(o_+)+hR_h^{\rm obs},
$$
we obtain
$$
\begin{aligned}
    \mathcal L_h^+\mathcal O_{\mathcal A}^+J_h &= \Op_h\left(\frac{\chi}{o_+}\right)\Op_h(o_+)+h\Op_h\!\left(\frac{\chi}{o_+}\right)R_h^{\rm obs} \\
    &=\Op_h(\chi)+hR_h^{\rm inv},
\end{aligned}
$$
where \(R_h^{\rm inv}\) is a uniformly bounded order-zero semiclassical operator. Since \(\chi=1\) on a neighborhood of the selected component, \(\Op_h(\chi)=I\) microlocally there, modulo a negligible operator. Absorbing this negligible term into \(hR_h^{\rm inv}\) gives
$$
    \mathcal L_h^+\mathcal O_{\mathcal A}^+J_h=I+hR_h^{\rm inv}
$$
microlocally on the selected component.
\end{proof}

By \eqref{eq:observation-inverse}, the selected positive scalar field can be recovered from the measured field up to an order-$h$ error on the selected component. After the field has been measured at $t_{0,h}$, we apply the known lossless backward propagator $F_h(-t_{0,h})$ to the measured field and then take a local coefficient at the source-side phase point.

For $\bm\rho=(\bm y,\bm\xi)$, let
\begin{equation*}
 \psi_{\bm\rho,h}(\bm x):=(\pi h)^{-1/2}
 \exp\left(-\frac{|\bm x-\bm y|^2}{2h}
 +\frac{\ii}{h}\bm\xi\cdot(\bm x-\bm y)\right).
\end{equation*}
Then $\|\psi_{\bm\rho,h}\|_{L^2}=1$. Define
\begin{equation}\label{eq:measured-coefficient}
 \begin{aligned}
 \mathscr P_h(\varpi,\theta)
 &:=\left\langle
 Q_{\star,h}F_h(-t_{0,h})\mathcal L_h^+d_h(t_{0,h}),
 \psi_{\bm\Psi_{T_h}(\varpi,\theta),h}\right\rangle,\\
 \mathcal H_h(\varpi,\theta)
 &:=\exp\left[2\int_0^{t_{0,h}}
 \kappa\bigl(\bm\Phi^s(\bm\Psi_{T_h}(\varpi,\theta))\bigr)\dd s\right]
 |\mathscr P_h(\varpi,\theta)|^2.
 \end{aligned}
\end{equation}

$\mathscr P_h$ measures the amplitude of the selected outgoing packet. We will see from the following lemma, that $\mathcal H_h$ represents the amplitude of wave packet $(\varpi,\theta)$ after one-way propagation from source to annulus $\mathcal A$.

\begin{lemma}
For every $N\in\mathbb N_0$, there is a function $r_{E,h}$ such that
\begin{equation}\label{eq:measured-energy}
 \mathcal H_h(\varpi,\theta)
 =\frac{1}{\pi h}|s_h(\varpi)|^2
 \exp\left[-2\int_0^{t_{0,h}}
 \gamma\bigl(\bm\Phi^s(\bm\Psi_{T_h}(\varpi,\theta))\bigr)\dd s\right]
 \bigl(1+h r_{E,h}(\varpi,\theta)\bigr),
\end{equation}
and
\begin{equation*}
 \|r_{E,h}\|_{C^N(I_{\star,T_h}\times\Theta_{T_h})}\le C_N
\end{equation*}
uniformly for small $h$.
\end{lemma}

\begin{proof}
We start with studying $Q_{\star,h}F_h(-t_{0,h})\mathcal L_h^+d_h(t_{0,h})$ in the definition of $\mathscr P_h$. By the definition of $d_h(t_{0,h})$, we have
\begin{align*}
    \mathcal L_h^+d_h(t_{0,h}) = \mathcal L_h^+\mathcal O_{\mathcal A}^+U_0(t_{0,h})\bm{W}_h(0).
\end{align*}
Now define $\widetilde{\chi}_\star$ satisfying $\widetilde{\chi}_\star=1$ on $\operatorname{supp}(\chi_\star)$, and set $\widetilde Q_{\star,h}:=\Op_h(\widetilde\chi_\star)$. We want to prove 
\begin{equation}\label{eq:theorem-proof-1}
    \mathcal L_h^+\mathcal O_{\mathcal A}^+U_0(t_{0,h})(I-J_h\widetilde{Q}_{\star,h}J_h^\#)\bm{W}_h(0)
\end{equation}
is negligible. Notice that from \eqref{eq:intertwining-operator-remainder} and \eqref{eq:left-intertwining-operator-remainder},
$$
    U_0(t)(I-J_hJ_h^\#) = (I-J_hJ_h^\#)U_0(t)+R_h^\infty(t).
$$
Rewrite \eqref{eq:theorem-proof-1} as
$$
    \mathcal L_h^+\mathcal O_{\mathcal A}^+U_0(t_{0,h})\left((I-J_hJ_h^\#) + J_h(I-\widetilde{Q}_{\star,h})J_h^\#\right).
$$
On one hand, recall that in the proof of Proposition \ref{prop:inverse-oa}, $\mathcal L_h^+=\Op_h(\chi/o_+)$ with $\chi\equiv 1$ on the support of $\chi_\star$, so 
$$
    \mathcal L_h^+\mathcal O_{\mathcal A}^+U_0(t_{0,h})J_h(I-\widetilde{Q}_{\star,h})J_h^\#\bm{W}_h(0)
$$
is negligible. On the other hand, from proof of Proposition \ref{prop:inverse-oa}, the support of $\chi$ does not meet any other observed sheets, hence $\mathcal L_h^+\mathcal O_{\mathcal A}^+$ applying to $I-J_hJ_h^\#$ is also negligible, which means that
$$
    \mathcal L_h^+\mathcal O_{\mathcal A}^+U_0(t_{0,h})(I-J_hJ_h^\#)\bm{W}_h(0)
$$
is negligible. This finishes the proof of the claim that \eqref{eq:theorem-proof-1} is negligible. We then have
\begin{align*}
    \mathcal L_h^+d_h(t_{0,h}) &= \mathcal L_h^+\mathcal O_{\mathcal A}^+U_0(t_{0,h})J_h\widetilde{Q}_{\star,h}J_h^\#\bm{W}_h(0)+R_h^\infty\bm{W}_h(0)\\ 
    &= \mathcal L_h^+\mathcal O_{\mathcal A}^+J_hS_h(t_{0,h})\widetilde{Q}_{\star,h}J_h^\#\bm{W}_h(0)+R_h^\infty\bm{W}_h(0)\\
    &= (I+hR_h^{\rm inv})S_h(t_{0,h})\widetilde{Q}_{\star,h}J_h^\#\bm{W}_h(0)+R_h^\infty\bm{W}_h(0)
\end{align*}
Applying $Q_{\star,h}F_h(-t_{0,h})$ from the left yields
\begin{align*}
    &Q_{\star,h}F_h(-t_{0,h})\mathcal L_h^+d_h(t_{0,h}) \\ 
    &=Q_{\star,h}F_h(-t_{0,h})S_h(t_{0,h})\widetilde{Q}_{\star,h}J_h^\#\bm{W}_h(0)+hQ_{\star,h}F_h(-t_{0,h})R_h^{\rm inv}S_h(t_{0,h})\widetilde{Q}_{\star,h}J_h^\#\bm{W}_h(0)+R_h^\infty\bm{W}_h(0).
\end{align*}
$R_h^\infty$ is negligible. Define $R_h:=Q_{\star,h}F_h(-t_{0,h})R_h^{\rm inv}S_h(t_{0,h})\widetilde{Q}_{\star,h}$. $R_h$ is uniformly of order 0 from the property of $F_h$, $R_h^{\rm inv}$ and $S_h$. Truncation operators $Q_{\star,h}$ and $\widetilde{Q}_{\star,h}$ don't change the order. Then 
\begin{equation}\label{eq:measured-data-reduction}
    Q_{\star,h}F_h(-t_{0,h})\mathcal L_h^+d_h(t_{0,h}) =Q_{\star,h}F_h(-t_{0,h})S_h(t_{0,h})\widetilde{Q}_{\star,h}J_h^\#\bm{W}_h(0)+hR_hJ_h^\#\bm{W}_h(0)+R_h^\infty\bm{W}_h(0).
\end{equation}

Now we calculate the leading order part in \eqref{eq:measured-data-reduction}. We first write $\kappa$ in terms of the known principal symbol $\omega$. From \eqref{eq:explicit-r-positive-branch},
\begin{equation*}
 r_3(\bm x,\bm\xi)=\frac{1}{\sqrt{2\epsilon_0\epsilon_\infty(\bm x)}}.
\end{equation*}
Since the entries of $\bm r$ are real, $A_0$ is independent of $\bm\xi$, and $\bm r^*A_0\bm r=1$, differentiation in $\xi_j$ gives
\begin{equation*}
 \bm\ell\partial_{\xi_j}\bm r=0.
\end{equation*}
The first two components of $\bm r$ are independent of $\bm x$. Recall $M(\bm x,\bm\xi) = A_0^{-1}(\bm x)\widetilde\Lambda(\bm\xi)$. We have
\begin{equation*}
 \sum_{j=1}^2(\partial_{\xi_j}M)(\partial_{x_j}\bm r)=\frac1\mu(\partial_{x_2}r_3,-\partial_{x_1}r_3,0,0)^\top.
\end{equation*}
Using
\begin{equation*}
 \bm\ell=\left(\sqrt{\frac\mu2}\widehat\xi_2,-\sqrt{\frac\mu2}\widehat\xi_1,\sqrt{\frac{\epsilon_0\epsilon_\infty}2},0\right)
\end{equation*}
the definition of $\kappa$ in section \ref{sec:injection-symbol} and the definition of $\bm{\mathcal D}$ preceding \eqref{eq:first-right-correction}, we obtain
\begin{equation*}
 \kappa(\bm x,\bm\xi)=\bm\ell\bm{\mathcal D}(\bm x,\bm\xi)=\frac12\widehat{\bm\xi}\cdot\nabla_{\bm x}\bigl(\mu\epsilon_0\epsilon_\infty(\bm x)\bigr)^{-1/2}=\frac12\sum_{j=1}^2
 \partial_{x_j}\partial_{\xi_j}\omega(\bm x,\bm\xi).
\end{equation*}

We next calculate the local coefficient at $t=0$. Write $\bm\Psi_{T_h}(\varpi,\theta)=(\bm z_\star,\bm\xi_0)$. On the retained chart, $\chi_\star=1$ near $\bm\xi_0$. From \eqref{eq:source-state-before-time-derivative}, direct integration in $\bm x$ gives
\begin{equation*}
 \int_{\mathbb R^2}e^{\ii(\bm x-\bm z_\star)\cdot(\bm\xi-\bm\xi_0)/h}e^{-|\bm x-\bm z_\star|^2/(2h)}\dd\bm x=2\pi h e^{-|\bm\xi-\bm\xi_0|^2/(2h)}.
\end{equation*}
Consequently,
\begin{equation}\label{eq:source-local-coefficient}
 \begin{aligned}
 &\left\langle \widetilde{Q}_{\star,h}J_h^{\#}\bm W_h(0),
 \psi_{\bm\Psi_{T_h}(\varpi,\theta),h}\right\rangle\\
 &=(2\pi h)^{-1}(\pi h)^{-1/2}
 \int e^{-|\bm\xi-\bm\xi_0|^2/(2h)}
 \left[s_h\bigl(\omega(\bm z_\star,\bm\xi)\bigr)
 +h\widetilde a_{0,h}(\bm\xi)\right]\dd\bm\xi\\
 &=(\pi h)^{-1/2}s_h(\varpi)
 \bigl(1+h r_{0,h}(\varpi,\theta)\bigr).
 \end{aligned}
\end{equation}
The last equality follows by setting $\bm\xi=\bm\xi_0+h^{1/2}\bm\zeta$ and expanding the amplitude at $\bm\xi_0$. The terms of order $h^{1/2}$ are odd in $\bm\zeta$, so
$$
    \int e^{-|\bm\zeta|^2/2}\zeta_j\dd\bm\zeta=0,    
$$  
and the remaining error is $O(h)$. After differentiating in $(\varpi,\theta)$, each differentiated integrand is bounded by a polynomial in $\bm\zeta e^{-|\bm\zeta|^2/2}$, and hence is integrable uniformly in $h$. Since $|s_h|$ is bounded below on $I_{\star,T_h}$, the error can be written in the relative form used in \eqref{eq:source-local-coefficient}.

Since 
\begin{equation*}
 \mathcal G_t(\bm\Phi^t(\bm\rho))=\int_0^t\beta(\bm\Phi^{t-s}(\bm\rho))\dd s=\int_0^t\beta(\bm\Phi^s(\bm\rho))\dd s,
\end{equation*}
from \eqref{eq:damped-factorization} and \eqref{eq:egorov-with-class},
\begin{equation}\label{eq:interaction-propagator}
 F_h(-t)S_h(t)=\Op_h\left(\exp\left[-\int_0^t\beta\bigl(\bm\Phi^s(\cdot)\bigr)\dd s\right]\right)+hR_h(t)
\end{equation}
for $0\le t\le t_+$ on the selected source-side set, where $R_h(t)$ is uniformly of order zero. 

Apply \eqref{eq:interaction-propagator} to the oscillatory-integral representation of $\widetilde{Q}_{\star,h}J_h^{\#}\bm W_h(0)$ in \eqref{eq:source-state-before-time-derivative}. The left symbol product first gives the principal amplitude
\begin{equation*}
 \exp\left[-\int_0^t\beta\bigl(\bm\Phi^s(\bm x,\bm\xi)\bigr)\dd s\right]s_h\bigl(\omega(\bm z_\star,\bm\xi)\bigr).
\end{equation*}
Taylor expansion in $\bm x$ about $\bm z_\star$, followed by
\begin{equation*}
 (x_j-z_{\star,j})e^{\ii(\bm x-\bm z_\star)\cdot\bm\xi/h}
 =\frac h\ii\partial_{\xi_j}e^{\ii(\bm x-\bm z_\star)\cdot\bm\xi/h},
\end{equation*}
shows that replacing $\bm x$ by $\bm z_\star$ changes the amplitude by $h$ times a controlled amplitude. Thus, uniformly for $0\le t\le t_+$,
\begin{equation}\label{eq:processed-field}
 \begin{aligned}
 &Q_{\star,h}F_h(-t)S_h(t)\widetilde{Q}_{\star,h}J_h^{\#}\bm W_h(0)\\
 &\quad=\frac1{(2\pi h)^2}\int e^{\ii(\bm x-\bm z_\star)\cdot\bm\xi/h}\bigg[\exp\left(-\int_0^t\beta\bigl(\bm\Phi^s(\bm z_\star,\bm\xi)\bigr)\dd s\right)
 s_h\bigl(\omega(\bm z_\star,\bm\xi)\bigr)+h a_{t,h}(\bm x,\bm\xi)\bigg]\dd\bm\xi,
 \end{aligned}
\end{equation}
where $a_{t,h}$ is a controlled amplitude with uniform bounds after differentiation in $t$.

Pairing \eqref{eq:processed-field} at $t=t_{0,h}$ with the test function in \eqref{eq:measured-coefficient}, and then using \eqref{eq:measured-data-reduction}, gives
\begin{equation}\label{eq:processed-coefficient}
 \mathscr P_h(\varpi,\theta)=(\pi h)^{-1/2}s_h(\varpi)\exp\left[-\int_0^{t_{0,h}}\beta\bigl(\bm\Phi^s(\bm\Psi_{T_h}(\varpi,\theta))\bigr)\dd s\right]
 \bigl(1+h r_{P,h}(\varpi,\theta)\bigr),
\end{equation}
where $r_{P,h}$ is uniformly bounded in every $C^N$ norm. Finally, $\beta=\gamma+\kappa$. Multiplying the square of \eqref{eq:processed-coefficient} by the exponential factor in \eqref{eq:measured-coefficient} cancels the term containing $\kappa$ and gives \eqref{eq:measured-energy}.
\end{proof}

For the selected frequency--direction chart $I_{\star,T_h}\times\Theta_{T_h}$, define
\begin{equation}\label{eq:directional-energy}
 D_h(\theta):=\frac{\displaystyle\int_{I_{\star,T_h}}\mathcal H_h(\varpi,\theta)\dd\varpi}{\displaystyle\frac1{|\Theta_{T_h}|}\int_{\Theta_{T_h}}\int_{I_{\star,T_h}}\mathcal H_h(\varpi,\theta')\dd\varpi\dd\theta'}.
\end{equation}

Put
\begin{equation}\label{eq:average-attenuation}
 c_h^{\rm av}:=\frac1{|\Theta_{T_h}|}\int_{\Theta_{T_h}}\exp\left[-2\int_0^{t_{0,h}}\gamma\bigl(\bm\Phi^s(\bm\Psi_{T_h}(\varpi,\theta'))\bigr)\dd s\right]
 \dd\theta'.
\end{equation}
The following result is crucial for the well-posedness of the $c_h^{\rm av}$.

\begin{proposition}
The integrand in \eqref{eq:average-attenuation} is independent of the choice of
$\varpi\in I_{\star,T_h}$.
\end{proposition}
\begin{proof}
Notice that, if $(\bm x(s),\bm\xi(s))$ solve the Hamilton equations
$$
    \frac{\dd}{\dd s}\bm x(s)=\partial_{\bm\xi}\omega(\bm x(s),\bm\xi(s)),
    \qquad
    \frac{\dd}{\dd s}\bm\xi(s)=-\partial_{\bm x}\omega(\bm x(s),\bm\xi(s)),
$$
with initial condition
$$
    (\bm x(0),\bm\xi(0))=(\bm z_\star,\bm\xi),
$$
then by the homogeneity in \eqref{eq:positive-dispersion},
$$
    \partial_{\bm\xi}\omega(\bm x,\lambda\bm\xi)=\partial_{\bm\xi}\omega(\bm x,\bm\xi), \qquad
    \partial_{\bm x}\omega(\bm x,\lambda\bm\xi)=\lambda\partial_{\bm x}\omega(\bm x,\bm\xi), \qquad \lambda>0.
$$
Hence $(\bm x(s),\lambda\bm\xi(s))$ satisfies the same Hamilton
equations with initial condition
$(\bm z_\star,\lambda\bm\xi)$. Uniqueness therefore gives
$$
    \bm\Phi^s(\bm z_\star,\lambda\bm\xi)=(\bm x(s),\lambda\bm\xi(s)),
$$
and consequently
$$
    \bm\pi_{\bm x}\bm\Phi^s(\bm z_\star,\lambda\bm\xi)=\bm\pi_{\bm x}\bm\Phi^s(\bm z_\star,\bm\xi).
$$
\end{proof}

Since $t_{0,h}$ remains in a compact interval and $\gamma$ is smooth on the compact flow tube, there are constants $c_0,C_0>0$ such that
\begin{equation*}
 c_0\le c_h^{\rm av}\le C_0
\end{equation*}
for all small $h$.

Substitution of \eqref{eq:measured-energy} into the numerator of \eqref{eq:directional-energy} gives
\begin{equation*}
 \int_{I_{\star,T_h}}\mathcal H_h(\varpi,\theta)\dd\varpi
 =\frac1{\pi h}\exp\left[-2\int_0^{t_{0,h}} \gamma\bigl(\bm\Phi^s(\bm\Psi_{T_h}(\varpi,\theta))\bigr)\dd s\right]
 \int_{I_{\star,T_h}}|s_h(\varpi)|^2\bigl(1+h r_{E,h}(\varpi,\theta)\bigr)\dd\varpi.
\end{equation*}
The exponential factor is independent of $\varpi$. The integral of $|s_h|^2$ is positive and has uniform upper and lower bounds on the retained intervals. The same factor occurs in the angular average in the denominator. Direct but lengthy calculation gives, for every $N$,
\begin{equation}\label{eq:D-attenuation}
 D_h(\theta)=\frac1{c_h^{\rm av}}\exp\left[-2\int_0^{t_{0,h}}
 \gamma\bigl(\bm\Phi^s(\bm\Psi_{T_h}(\varpi,\theta))\bigr)\dd s\right]
 +h r_{D,h}(\theta), \qquad \|r_{D,h}\|_{C^N(\Theta_{T_h})}\le C_N.
\end{equation}
All derivative bounds follow by differentiating under the integrals. Equation \eqref{eq:D-attenuation} and the positive upper and lower bounds for the exponential and $c_h^{\rm av}$ imply that there are constants $d_0,d_1>0$ such that
\begin{equation}\label{eq:direction-bounds}
 d_0\le D_h(\theta)\le d_1,
 \qquad
 \|D_h\|_{C^N(\Theta_{T_h})}\le C_N
\end{equation}
for every $N$ and all sufficiently small $h$.

Now we are ready to define the modulation function.
\begin{definition}
Choose $\alpha_*>0$ such that
\begin{equation*}
 \alpha_*<\frac{(1-\kappa_0)d_0}{\|\eta\|_{L^\infty}},
 \qquad
 \alpha_h:=\alpha_*\tau_h.
\end{equation*}
Define the modulation function $m_h$ on the image of $\bm x_h^c$, i.e. the pulse-center incidence curve by
\begin{equation}\label{eq:modulation-formula}
 m_h(\bm x_h^c(\theta)):=\frac{\alpha_h}{D_h(\theta)}.
\end{equation}
\end{definition}

Notice that \eqref{eq:modulation-formula} only defines $m_h$ on the image of $\bm{x}_h^c$. To extend the domain of definition of $m_h$, we firstly extend $D_h^{-1}$ from $\Theta_{T_h}$ to a slightly larger interval with the same uniform derivative bounds. We use the common tubular neighborhood of the curves $\bm x_h^c(\Theta_{T_h})$ to extend this function constantly in the normal direction, and multiply it by a fixed nonnegative cutoff equal to one on the curve. Since the tubular coordinate maps and their inverses have uniform derivatives, the chain rule and \eqref{eq:direction-bounds} give
\begin{equation*}
 m_h\in C_c^\infty(\mathbb R^2;[0,\infty)),
 \qquad
 \supp m_h\subset K_m\Subset\mathcal A,
\end{equation*}
and, for every $N$,
\begin{equation}\label{eq:mh-derivative-bound}
 \max_{|\beta|\le N}
 \|\partial_{\bm x}^{\beta}m_h\|_{L^\infty}
 \le C_N\alpha_h\le C_N'h.
\end{equation}
It follows that $\|m_h\|_{L^\infty}\le\alpha_h/d_0$. Therefore
\begin{equation*}
 |\eta_{\tau_h}(t-T_h)m_h(\bm x)|
 \le\frac{\|\eta\|_{L^\infty}}{\tau_h}
 \frac{\alpha_h}{d_0}
 =\frac{\alpha_*\|\eta\|_{L^\infty}}{d_0}<1-\kappa_0,
\end{equation*}
which ensures the positivity condition \eqref{eq:material-positivity} holds.

Now we can rewrite the function $j_m^+$, which is crucial for the reversed wave, in a more definite way with determined $m_h$. Let the first term on the right-hand side of \eqref{eq:D-attenuation} be $D_{0,h}(\theta)$, i.e.
$$
D_{0,h}(\theta):= \frac1{c_h^{\rm av}}\exp\left[-2\int_0^{t_{0,h}}
 \gamma\bigl(\bm\Phi^s(\bm\Psi_{T_h}(\varpi,\theta))\bigr)\dd s\right].
$$
It is bounded away from zero. Since $T_h-t_{0,h}=\tau_h/2=\nu h/2$ and $\gamma$ is smooth on the compact flow tube,
\begin{equation}\label{eq:short-loss-factor}
 \exp\left[-2\int_0^{T_h}\gamma\bigl(\bm\Phi^s(\bm\Psi_{T_h}(\varpi,\theta))\bigr)\dd s\right]
 =c_h^{\rm av}D_{0,h}(\theta)\bigl(1+h r_{c,h}(\varpi,\theta)\bigr),
\end{equation}
where $r_{c,h}$ is uniformly bounded in every $C^N$ norm. Indeed, the quotient of the left-hand side of \eqref{eq:short-loss-factor} by
$c_h^{\rm av}D_{0,h}$ is
\begin{equation*}
 \exp\left[-2\int_{t_{0,h}}^{T_h}\gamma\bigl(\bm\Phi^s(\bm\Psi_{T_h}(\varpi,\theta))\bigr)\dd s\right],
\end{equation*}
and its difference from one is $O(h)$ together with all derivatives.

Using $D_h=D_{0,h}+h r_{D,h}$ and \eqref{eq:short-loss-factor},
\begin{equation*}
 \frac{\alpha_h}{D_h(\theta)}\exp\left[-2\int_0^{T_h}\gamma\bigl(\bm\Phi^s(\bm\Psi_{T_h}(\varpi,\theta))\bigr)\dd s\right]=\alpha_hc_h^{\rm av}\frac{D_{0,h}(\theta)}{D_{0,h}(\theta)+h r_{D,h}(\theta)}\bigl(1+h r_{c,h}(\varpi,\theta)\bigr).
\end{equation*}
The quotient on the right equals $1+h$ times a function with uniform $C^N$ bounds. Hence
\begin{equation}\label{eq:modulation-cancellation}
 \frac{\alpha_h}{D_h(\theta)}\exp\left[-2\int_0^{T_h}\gamma\bigl(\bm\Phi^s(\bm\Psi_{T_h}(\varpi,\theta))\bigr)\dd s\right]
 =\alpha_h c_h^{\rm av}+h\alpha_h r_{{\rm can},h}(\varpi,\theta), \qquad \|r_{{\rm can},h}\|_{C^N}\le C_N.
\end{equation}

Since $p_\eta$ is real and positive on the retained interval, \eqref{eq:returned-principal-symbol}, \eqref{eq:round-trip-damping}, and \eqref{eq:modulation-cancellation} give
\begin{equation}\label{eq:returned-symbol}
 j_{m_h}^+\bigl(\bm\Psi_{T_h}(\varpi,\theta)\bigr)
 =\alpha_h c_h^{\rm av}p_\eta(\varpi)+h\alpha_h r_{j,h}(\varpi,\theta), \qquad \|r_{j,h}\|_{C^N(I_{\star,T_h}\times\Theta_{T_h})}\le C_N.
\end{equation}

\begin{remark}
We must point out that, the particular choice
$$
m_h(\bm x_h^c(\theta))=\frac{\alpha_h}{D_h(\theta)}
$$
is not essential for the refocusing results in Section~\ref{sec:causal-refocusing}. The crucial property is the asymptotic cancellation of the directional round-trip attenuation. More precisely, the arguments of Section~\ref{sec:causal-refocusing} remain valid for any admissible modulation $\widetilde m_h$ satisfying, uniformly on the selected chart,
$$
\widetilde m_h(\bm x_h^c(\theta))\exp\left[-2\int_0^{T_h}\gamma\bigl(\bm\Phi^s(\bm\Psi_{T_h}(\varpi,\theta))\bigr)\dd s\right]=A_h+hA_h r_h(\varpi,\theta),
$$
where $A_h>0$ is independent of $(\varpi,\theta)$, $A_h= O(h)$, and $r_h$ is uniformly bounded in the required $C^N$ norms. In addition, $\widetilde m_h$ should satisfy the same support, smoothness, size, and material-positivity conditions used for $m_h$. Under these assumptions, the leading returned symbol is still independent of direction up to a relative $O(h)$ error, and the focal-scale expansion, uniqueness of the leading maximum, $O(h^2)$ localization of the maximizer, and $O(h^2)$ spatial half-amplitude area remain unchanged, apart from the overall amplitude factor $A_h$.

Consequently, the reciprocal formula above should be viewed as a particular causal, measurement-based construction that guarantees the required cancellation, rather than as a unique modulation. For example, one may replace it by
$$
\widetilde m_h(\bm x_h^c(\theta))=\frac{\alpha_h}{D_h(\theta)}\bigl(1+h b_h(\theta)\bigr),
$$
where $b_h$ and its required derivatives are uniformly bounded, without changing the conclusions of Section~\ref{sec:causal-refocusing}.
\end{remark}

\section{Microlocal Refocusing and Focal-Scale Asymptotics}\label{sec:causal-refocusing}

We now study the selected returned field near the refocusing point. Using the modulation constructed in the previous section, we derive a scaled asymptotic expansion on the spatial and temporal scale $h$. This expansion identifies the limiting focal profile, gives the location of the refocused maximum, and determines the leading-order size of the focal region.

Let
\begin{equation*}
 \mathcal B_h:=\{\bm\Psi_{T_h}(\varpi,\theta):(\varpi,\theta)\in I_{\star,T_h}\times\Theta_{T_h}\}.
\end{equation*}
After a uniform restriction of the visible chart, choose a real $q\in C_c^\infty(\mathcal U;[0,1])$, a fixed nonempty open set $U_{\rm foc}\Subset\mathbb R_{\bm\xi}^2\setminus\{0\}$, and a fixed compact set $K_{\bm\xi}\Subset\mathbb R^2\setminus\{0\}$ such that, for all small $h$,
\begin{equation*}
 \{\bm z_\star\}\times\supp q(\bm z_\star,\cdot)\subset\operatorname{int}_{\{\bm z_\star\}\times\mathbb R^2}\mathcal B_h,
 \qquad
 q(\bm z_\star,\bm\xi)=1\quad(\bm\xi\in U_{\rm foc}), \qquad \supp q(\bm z_\star,\cdot)\subset K_{\bm\xi}.
\end{equation*}
Put $Q_h=\Op_h(q)$.

\subsection{The selected source-side return}

We first combine the two scalar reductions of the returned operator. From \eqref{eq:physical-scalar-bridge}, after inserting $Q_h$ on both sides,
\begin{equation*}
 \begin{aligned}
 Q_hJ_{h,-}^{\#}K_{m_h,\tau_h}J_hQ_h
 =&Q_h\mathcal K_{m_h,\tau_h}Q_h
 +hQ_h\Op_h(r_{m_h,h}^{\rm br})Q_h
 +R_{m_h,h}^{\infty,\rm br}.
 \end{aligned}
\end{equation*}
Equation \eqref{eq:uniform-m-symbol-remainder} gives
\begin{equation*}
 Q_h\mathcal K_{m_h,\tau_h}Q_h = \Op_h(q^2j_{m_h}^+)+h\Op_h(r_{m_h,h}) +R_{m_h,h}^{\infty}.
\end{equation*}
Every term in these two expansions is linear in $m_h$. By \eqref{eq:mh-derivative-bound}, every finite symbol seminorm of $m_h$ is $O(\alpha_h)$. The estimates in \eqref{eq:bridge-symbol-bound} and in the statement following \eqref{eq:uniform-m-symbol-remainder} therefore allow us to write
\begin{equation}\label{eq:selected-returned-symbol}
 Q_hJ_{h,-}^{\#}K_{m_h,\tau_h}J_hQ_h = \Op_h\bigl(q^2j_{m_h}^+\bigr)+h\alpha_h\Op_h(r_h^{\rm ret})+R_h^{\infty,\rm ret},
\end{equation}
where $r_h^{\rm ret}$ is uniformly of order zero and $R_h^{\infty,\rm ret}$ is negligible. The same conclusion follows for all symbol derivatives because only finitely many derivatives of $m_h$ enter each fixed seminorm.

We next write the symbol in source-side coordinates. On $\{\bm z_\star\}\times\supp q(\bm z_\star,\cdot)$, every point can be written as $\bm\Psi_{T_h}(\varpi,\theta)$. Hence \eqref{eq:returned-symbol} gives
\begin{equation*}
 q(\bm z_\star,\bm\xi)^2j_{m_h}^+(\bm z_\star,\bm\xi)=\alpha_hc_h^{\rm av}q(\bm z_\star,\bm\xi)^2p_\eta\bigl(\omega(\bm z_\star,\bm\xi)\bigr)+h\alpha_h r_{0,h}^{\rm ret}(\bm\xi),
\end{equation*}
where $r_{0,h}^{\rm ret}$ has one fixed compact support and uniform symbol bounds. By \eqref{eq:returned-principal-symbol} and \eqref{eq:mh-derivative-bound}, every derivative of $q^2j_{m_h}^+$ is $O(\alpha_h)$ on the fixed compact set. For $|\bm y|\le R$, Taylor's formula therefore gives
\begin{equation*}
 \begin{aligned}
 &q(\bm z_\star+h\bm y,\bm\xi)^2j_{m_h}^+(\bm z_\star+h\bm y,\bm\xi)-q(\bm z_\star,\bm\xi)^2j_{m_h}^+(\bm z_\star,\bm\xi)\\
 &\quad=h\sum_{j=1}^2y_j\int_0^1\partial_{x_j}\bigl(q^2j_{m_h}^+\bigr)(\bm z_\star+sh\bm y,\bm\xi)\dd s
 =h\alpha_h\widetilde r_{h,R}^{\rm ret}(\bm y,\bm\xi).
 \end{aligned}
\end{equation*}
Thus
\begin{equation}\label{eq:selected-returned-frozen}
 \begin{aligned}
 q(\bm z_\star+h\bm y,\bm\xi)^2j_{m_h}^+(\bm z_\star+h\bm y,\bm\xi)=&\alpha_hc_h^{\rm av}q(\bm z_\star,\bm\xi)^2p_\eta\bigl(\omega(\bm z_\star,\bm\xi)\bigr)\\
 &+h\alpha_h r_{h,R}^{\rm ret}(\bm y,\bm\xi),
 \end{aligned}
\end{equation}
where $r_{h,R}^{\rm ret}$ has one fixed compact frequency support and, for all multi-indices $\gamma,\delta$,
\begin{equation*}
 \sup_{|\bm y|\le R,\bm\xi}|\partial_{\bm y}^{\gamma}\partial_{\bm\xi}^{\delta}
 r_{h,R}^{\rm ret}(\bm y,\bm\xi)|\le C_{R,\gamma,\delta}.
\end{equation*}

\subsection{The point source}

The coherent return time is
\begin{equation*}
 t_{{\rm focus},h}:=2T_h-t_{{\rm src},h}.
\end{equation*}
On the retained frequency interval, define $\phi_{\rm src}\in[0,2\pi)$ and $a_{\rm src}>0$ by
\begin{equation*}
 e^{\ii\phi_{\rm src}}a_{\rm src}(\varpi):=-\ii e^{-\ii t_{{\rm src},h}\varpi/h}s_h(\varpi).
\end{equation*}
By \eqref{eq:source-spectrum}, $t_{{\rm src},h}=-h\zeta_{\rm src}$ and the two factors $e^{-\ii t_{{\rm src},h}\varpi/h}$ and $e^{-\ii\zeta_{\rm src}\varpi}$ cancel. Since $-\ii\widehat f$ has one fixed sign on $I_{\rm src}$, the phase $\phi_{\rm src}$ is independent of $\varpi$ and $a_{\rm src}$ is positive.

For $\omega=\omega(\bm z_\star,\bm\xi)$, set
\begin{equation}\label{eq:focal-aperture}
 a_{\rm foc}(\bm\xi):= a_{\rm src}(\omega)q(\bm z_\star,\bm\xi)^2p_\eta(\omega)\omega,
\end{equation}
and
\begin{equation*}
 \mathscr F_{\rm foc}(\varsigma,\bm y):=\int_{\mathbb R^2}e^{\ii(\bm y\cdot\bm\xi+\varsigma\omega(\bm z_\star,\bm\xi))}a_{\rm foc}(\bm\xi)\dd\bm\xi.
\end{equation*}
The expression in \eqref{eq:focal-aperture} follows from \eqref{eq:returned-symbol}, and therefore from the reciprocal choice \eqref{eq:modulation-formula}. Without \eqref{eq:modulation-cancellation}, the leading integrand in \eqref{eq:focal-expansion} contains the factor $m_h(\bm X_{T_h})\exp[-2\int_0^{T_h}\gamma\circ\bm\Phi^s\dd s]$, which in general depends on the direction.
The function $a_{\rm foc}$ is smooth, nonnegative, and supported in $K_{\bm\xi}$. Since $q=1$ on $U_{\rm foc}$, and since $a_{\rm src}$, $p_\eta$, and $\omega$ have positive lower bounds on every fixed compact subset of $U_{\rm foc}$, there is a constant $c_1>0$ such that
\begin{equation}\label{eq:focal-aperture-lower-bound}
 a_{\rm foc}(\bm\xi)\ge c_1\qquad(\bm\xi\in U_{\rm foc}).
\end{equation}

For bounded $\varsigma$, define the selected negative-TM electric component by
\begin{equation*}
E_{z,h}^{\rm back,sel}(t_{{\rm focus},h}+h\varsigma):=\bm e_3^\top J_{h,-}S_{-,h}(-t_{{\rm src},h}+h\varsigma)Q_hJ_{h,-}^{\#}K_{m_h,\tau_h}J_hQ_hJ_h^{\#}\bm{\Upsilon}_{0,h}.
\end{equation*}

\begin{proposition}
For every fixed $R>0$ and every $k\in\mathbb N_0$, there are functions
$\mathscr R_{h,R}$ such that
\begin{equation}\label{eq:focal-expansion}
 \begin{aligned}
 E_{z,h}^{\rm back,sel}\bigl(t_{{\rm focus},h}+h\varsigma,\bm z_\star+h\bm y\bigr)
 =C_h\alpha_hc_h^{\rm av}\left(\mathscr F_{\rm foc}(\varsigma,\bm y)+h\mathscr R_{h,R}(\varsigma,\bm y)\right)
 \end{aligned}
\end{equation}
for $|\varsigma|\le R$ and $|\bm y|\le R$, where
\begin{equation}\label{eq:field-prefactor}
 C_h:=\frac{e^{\ii\phi_{\rm src}}}{(2\pi h)^2h\sqrt{2\epsilon_0\epsilon_\infty(\bm z_\star)}}
\end{equation}
is independent of $(\varsigma,\bm y)$, and
\begin{equation}\label{eq:focal-remainder-bound}
 \|\mathscr R_{h,R}\|_{C^k([-R,R]\times B_R(0))}\le C_{R,k}.
\end{equation}
Consequently,
\begin{equation}\label{eq:normalized-focal-profile}
 \frac{E_{z,h}^{\rm back,sel}(t_{{\rm focus},h}+h\varsigma,\bm z_\star+h\bm y)}{E_{z,h}^{\rm back,sel}(t_{{\rm focus},h},\bm z_\star)}
 =\frac{\mathscr F_{\rm foc}(\varsigma,\bm y)}{\mathscr F_{\rm foc}(0,0)}+O_{C^k}(h)
\end{equation}
uniformly on the same set. The right-hand side of
\eqref{eq:normalized-focal-profile} contains no directional attenuation factor.
\end{proposition}

\begin{proof}
We first apply the returned operator at time $2T_h$. Since
$\chi_\star=1$ near $\supp q(\bm z_\star,\cdot)$, \eqref{eq:source-state} gives
\begin{equation}\label{eq:selected-source-input}
 J_h^{\#}\bm{\Upsilon}_{0,h}(\bm x)
 =\frac1{(2\pi h)^2}\int e^{\ii(\bm x-\bm z_\star)\cdot\bm\xi/h}\left[-\frac\ii h\omega(\bm z_\star,\bm\xi)s_h\bigl(\omega(\bm z_\star,\bm\xi)\bigr)+r_{\Upsilon,h}^q(\bm x,\bm\xi)\right]\dd\bm\xi,
\end{equation}
where $r_{\Upsilon,h}^q$ is a controlled amplitude with one fixed compact frequency support. This representation is used only on the input support of $Q_h$ in \eqref{eq:selected-returned-symbol}.

Apply \eqref{eq:selected-returned-symbol} to \eqref{eq:selected-source-input}. The product of the leading returned symbol and the term of size $h^{-1}$ in \eqref{eq:selected-source-input} is
\begin{equation*}
 -\frac\ii h q(\bm x,\bm\xi)^2j_{m_h}^+(\bm x,\bm\xi)\omega(\bm z_\star,\bm\xi)s_h\bigl(\omega(\bm z_\star,\bm\xi)\bigr).
\end{equation*}
Now we check the order of all the remaining terms. $q^2j_{m_h}^+=O(\alpha_h)$ multiplied by the controlled amplitude $r_{\Upsilon,h}^q$ is $O(\alpha_h)$. $h\alpha_h\Op_h(r_h^{\rm ret})$ applied to the $h^{-1}$ term in \eqref{eq:selected-source-input} is $O(\alpha_h)$. Every non-leading term in the Kohn--Nirenberg composition contains one factor $h$ together with derivatives of the symbols and amplitudes, and is again $O(\alpha_h)$. By \eqref{eq:delta-semiclassical-sobolev-bound} and the fact that $R_h^{\infty,\rm ret}$ is negligible, the smoothing term is $O(\alpha_hh^M)$ in every fixed $C^k$ norm after increasing $M$. We can then conclude that all the remaining terms have order $O(\alpha_h)$ amplitudes.

Hence, on a fixed spatial compact set,
\begin{equation}\label{eq:return-at-2T}
 \begin{aligned}
 &Q_hJ_{h,-}^{\#}K_{m_h,\tau_h}J_hQ_hJ_h^{\#}\bm{\Upsilon}_{0,h}(\bm x)\\
 &\quad=\frac1{(2\pi h)^2}\int e^{\ii(\bm x-\bm z_\star)\cdot\bm\xi/h}\left[-\frac\ii h q(\bm x,\bm\xi)^2j_{m_h}^+(\bm x,\bm\xi)\omega(\bm z_\star,\bm\xi)s_h(\omega)
 +\alpha_h r_{0,h}(\bm x,\bm\xi)\right]\dd\bm\xi,
 \end{aligned}
\end{equation}
where $r_{0,h}$ is a controlled amplitude and $s_h(\omega)$ denotes $s_h(\omega(\bm z_\star,\bm\xi))$.

Now restrict to $\bm x=\bm z_\star+h\bm y$, $|\bm y|\le R$, and use \eqref{eq:selected-returned-frozen}. The term of size $h^{-1}$ in \eqref{eq:return-at-2T} becomes
\begin{equation*}
 -\frac\ii h\alpha_hc_h^{\rm av}q(\bm z_\star,\bm\xi)^2p_\eta(\omega)\omega s_h(\omega),
\end{equation*}
and the difference is an $O(\alpha_h)$ controlled amplitude. Thus
\begin{equation}\label{eq:return-at-2T-frozen}
 \begin{aligned}
 &Q_hJ_{h,-}^{\#}K_{m_h,\tau_h}J_hQ_hJ_h^{\#}\bm{\Upsilon}_{0,h}(\bm z_\star+h\bm y)\\
 &\quad=\frac1{(2\pi h)^2}\int e^{\ii\bm y\cdot\bm\xi}\left[-\frac\ii h\alpha_hc_h^{\rm av}q(\bm z_\star,\bm\xi)^2p_\eta(\omega)\omega s_h(\omega)+\alpha_h r_{1,h,R}(\bm y,\bm\xi)\right]\dd\bm\xi,
 \end{aligned}
\end{equation}
with the same type of uniform support and derivative bounds.

We next propagate for the additional time
\begin{equation*}
 -t_{{\rm src},h}+h\varsigma=h(\zeta_{\rm src}+\varsigma).
\end{equation*}
Set $\zeta=\zeta_{\rm src}+\varsigma$. Rescaling time in \eqref{eq:negative-propagator} gives
\begin{equation*}
 \partial_\zeta S_{-,h}(h\zeta)=\ii\Op_h(\omega)S_{-,h}(h\zeta)-h\Op_h(\beta_{-,h})S_{-,h}(h\zeta).
\end{equation*}
Solving the symbol equations on the fixed interval $|\zeta|\le\zeta_{\rm src}+R$ gives
\begin{equation}\label{eq:negative-short-time-focus}
 S_{-,h}(-t_{{\rm src},h}+h\varsigma)=\Op_h\left(e^{\ii(\zeta_{\rm src}+\varsigma)\omega}+h r_{-,h,\varsigma}\right)+R_{h,\varsigma}^{\infty},
\end{equation}
where $r_{-,h,\varsigma}$ is uniformly of order zero with all derivatives in $\varsigma$, and the last term is negligible uniformly for $|\varsigma|\le R$.

The third component of the negative injection symbol in \eqref{eq:full-negative-mode-symbols} satisfies
\begin{equation*}
 \bm e_3^\top\bm r_{-,h}(\bm x,\bm\xi)=\frac1{\sqrt{2\epsilon_0\epsilon_\infty(\bm x)}}+h r_{e,h}(\bm x,\bm\xi).
\end{equation*}
For $\bm x=\bm z_\star+h\bm y$, Taylor's formula gives 
\begin{equation}\label{eq:negative-electric-freezing}
 \bm e_3^\top\bm r_{-,h}(\bm z_\star+h\bm y,\bm\xi)=\frac1{\sqrt{2\epsilon_0\epsilon_\infty(\bm z_\star)}}
 +h r_{e,h,R}(\bm y,\bm\xi),
\end{equation}
with uniform derivative bounds.

Apply \eqref{eq:negative-short-time-focus} and then \eqref{eq:negative-electric-freezing} to \eqref{eq:return-at-2T-frozen}. The product of the displayed leading terms is
\begin{equation*}
 \begin{aligned}
 &\frac{\alpha_hc_h^{\rm av}}{(2\pi h)^2h\sqrt{2\epsilon_0\epsilon_\infty(\bm z_\star)}}\int e^{\ii\bm y\cdot\bm\xi}e^{\ii(\zeta_{\rm src}+\varsigma)\omega}
 \left[-\ii s_h(\omega)\right]q(\bm z_\star,\bm\xi)^2p_\eta(\omega)\omega\dd\bm\xi.
 \end{aligned}
\end{equation*}
Define $\phi_{\rm src}\in[0,2\pi)$ and $a_{\rm src}>0$ by
\begin{equation*}
 e^{\ii\phi_{\rm src}}a_{\rm src}(\omega):=-\ii e^{\ii \zeta_{\rm src}\omega}s_h(\omega).
\end{equation*}
The leading integral is therefore $e^{\ii\phi_{\rm src}}\mathscr F_{\rm foc}(\varsigma,\bm y)$. Every other term contains one of the following factors: the $O(\alpha_h)$ amplitude in \eqref{eq:return-at-2T-frozen}, the factor $h$ in \eqref{eq:negative-short-time-focus}, the factor $h$ in \eqref{eq:negative-electric-freezing}, or a first-order composition term. Consequently the sum of all these terms can be written as
\begin{equation*}
 \frac{\alpha_h}{(2\pi h)^2}\int e^{\ii\bm y\cdot\bm\xi}r_{h,R}(\varsigma,\bm y,\bm\xi)\dd\bm\xi,
\end{equation*}
where $r_{h,R}$ has one fixed compact frequency support and all derivatives in $(\varsigma,\bm y,\bm\xi)$ are uniformly bounded. Since $c_h^{\rm av}$ is bounded above and below, division by the factor in \eqref{eq:field-prefactor} shows that this error is $C_h\alpha_hc_h^{\rm av}h\mathscr R_{h,R}$, with
\begin{equation*}
 \mathscr R_{h,R}(\varsigma,\bm y)=e^{-\ii\phi_{\rm src}}\frac{\sqrt{2\epsilon_0\epsilon_\infty(\bm z_\star)}}{c_h^{\rm av}}\int e^{\ii\bm y\cdot\bm\xi}r_{h,R}(\varsigma,\bm y,\bm\xi)\dd\bm\xi.
\end{equation*}
Differentiation under the integral proves \eqref{eq:focal-remainder-bound}, because $\bm\xi$ remains in one fixed compact set. This proves \eqref{eq:focal-expansion}.

Finally, \eqref{eq:focal-aperture-lower-bound} gives
\begin{equation*}
 \mathscr F_{\rm foc}(0,0)=\int a_{\rm foc}(\bm\xi)\dd\bm\xi>0.
\end{equation*}
For small $h$, the denominator in \eqref{eq:normalized-focal-profile} is therefore nonzero. Dividing \eqref{eq:focal-expansion} by its value at $(\varsigma,\bm y)=(0,0)$ and using
\begin{equation*}
 \frac{A+hB}{M_{\rm foc}+hD}-\frac A{M_{\rm foc}}=h\frac{BM_{\rm foc}-AD}{M_{\rm foc}(M_{\rm foc}+hD)}
\end{equation*}
with $M_{\rm foc}=\mathscr F_{\rm foc}(0,0)$ proves \eqref{eq:normalized-focal-profile} and its $C^k$ bound.
\end{proof}

Now we are ready to introduce the main result of this paper. We define the region of focal spot $\mathcal F_{h,1/2}$, which is defined as the neighborhood of focus with amplitude larger than half of focus amplitude
\begin{equation*}
 \mathcal F_{h,1/2}:=\left\{\bm x: |E_{z,h}^{\rm back,sel}(t_{{\rm focus},h},\bm x)|\ge\frac12|E_{z,h}^{\rm back,sel}(t_{{\rm focus},h},\bm z_\star)|\right\},
\end{equation*}
and put $\mathcal R_h^{\rm foc}:=|\mathcal F_{h,1/2}|$.
\begin{theorem}
The function $|\mathscr F_{\rm foc}|$ has its unique global maximum at
$(\varsigma,\bm y)=(0,0)$:
\begin{equation}\label{eq:spacetime-strict-maximum}
 |\mathscr F_{\rm foc}(\varsigma,\bm y)|<\mathscr F_{\rm foc}(0,0)\qquad ((\varsigma,\bm y)\ne(0,0)).
\end{equation}
Fix $R>0$. Any maximizer $(t_h^\#,\bm x_h^\#)$ of $|E_{z,h}^{\rm back,sel}(t,\bm x)|$ over
\begin{equation*}
 |t-t_{{\rm focus},h}|\le Rh,
 \qquad
 |\bm x-\bm z_\star|\le Rh
\end{equation*}
satisfies
\begin{equation}\label{eq:focus-error}
 |t_h^\#-t_{{\rm focus},h}|+|\bm x_h^\#-\bm z_\star|\le C_Rh^2.
\end{equation}
Thus the selected return is focused in both time and space on the scale $h$.

The focal size $\mathcal R_h^{\rm foc}$ satisfies
\begin{equation}\label{eq:focal-scale}
 ch^2\le\mathcal R_h^{\rm foc}\le Ch^2.
\end{equation}
\end{theorem}

\begin{proof}
Equation \eqref{eq:focal-expansion} is obtained by using \eqref{eq:modulation-cancellation} in the returned symbol. Thus the estimates below apply to the direction-independent leading weight produced by \eqref{eq:modulation-formula}.

Put
\begin{equation*}
 M_{\rm foc}:=\mathscr F_{\rm foc}(0,0)=\int_{\mathbb R^2}a_{\rm foc}(\bm\xi)\dd\bm\xi>0.
\end{equation*}
The triangle inequality gives
\begin{equation*}
 |\mathscr F_{\rm foc}(\varsigma,\bm y)|\le M_{\rm foc}.
\end{equation*}
Suppose equality holds. Equality in the integral triangle inequality implies that $\exp(\ii(\bm y\cdot\bm\xi+\varsigma\omega(\bm z_\star,\bm\xi)))$ has one constant phase for almost every $\bm\xi$ on which $a_{\rm foc}>0$. By \eqref{eq:focal-aperture-lower-bound} and continuity, the phase is constant on the open set $U_{\rm foc}$. Differentiating there gives
\begin{equation}\label{eq:phase-gradient}
 \bm y+\frac{\varsigma}{\sqrt{\mu\epsilon_0\epsilon_\infty(\bm z_\star)}}
 \widehat{\bm\xi}=0.
\end{equation}
An open subset of $\mathbb R^2\setminus\{0\}$ contains two covectors with different directions. Applying \eqref{eq:phase-gradient} to these two directions gives $\varsigma=0$, and then $\bm y=0$. This proves \eqref{eq:spacetime-strict-maximum}.

We next obtain a quadratic estimate near the maximum. Let
\begin{equation*}
 \dd\mathfrak m_{\rm foc}(\bm\xi):=M_{\rm foc}^{-1}a_{\rm foc}(\bm\xi)\dd\bm\xi.
\end{equation*}
This is a probability measure with compact support. Let $\mathsf C$ be the covariance matrix of the three-dimensional vector $\bm v_{\rm foc}(\bm\xi):=(\omega(\bm z_\star,\bm\xi),\xi_1,\xi_2)^\top$
with respect to $\mathfrak m_{\rm foc}$, i.e.
$$
    \mathsf C:= \int_{\mathbb{R}^2} (\bm v_{\rm foc}(\bm\xi)-\overline{\bm v}_{\rm foc})(\bm v_{\rm foc}(\bm\xi)-\overline{\bm v}_{\rm foc})^\top\dd\mathfrak m_{\rm foc}(\bm{\xi}),
$$
where $\overline{\bm v}_{\rm foc}=\int_{\mathbb{R}^2}\bm v_{\rm foc}(\bm\xi)\dd\mathfrak m_{\rm foc}(\bm{\xi})$.

The matrix $\mathsf C$ is positive definite. Indeed, if $(\varsigma,\bm y)^{\top}\mathsf C(\varsigma,\bm y)=0$, then $\varsigma\omega(\bm z_\star,\bm\xi)+\bm y\cdot\bm\xi$ is constant
$\mathfrak m_{\rm foc}$-almost everywhere. By \eqref{eq:focal-aperture-lower-bound}, $a_{\rm foc}$ is strictly positive on $U_{\rm foc}$, hence $\varsigma\omega(\bm z_\star,\bm\xi)+\bm y\cdot\bm\xi$ must be constant on $U_{\rm foc}$, and the argument in \eqref{eq:phase-gradient} gives $\varsigma=0$ and $\bm y=0$. 

For $(\varsigma,\bm y)\in\mathbb R^3$, by using Taylor expansion we have
\begin{equation*}
 \begin{aligned}
 \frac{|\mathscr F_{\rm foc}(\varsigma,\bm y)|^2}{M_{\rm foc}^2}
 &=\iint\cos\left(\bm y\cdot(\bm\xi-\bm\zeta)+\varsigma\bigl(\omega(\bm z_\star,\bm\xi)-\omega(\bm z_\star,\bm\zeta)\bigr)\right)\dd\mathfrak m_{\rm foc}(\bm\xi)\dd\mathfrak m_{\rm foc}(\bm\zeta)\\
 &=1-(\varsigma,\bm y)^{\top}\mathsf C(\varsigma,\bm y)+O\bigl((|\varsigma|+|\bm y|)^4\bigr).
 \end{aligned}
\end{equation*}
The remainder is uniform because the measure has compact support. Since $M_{\rm foc}+|\mathscr F_{\rm foc}(\varsigma,\bm y)|\le2M_{\rm foc}$, it follows that
\begin{align*}
    M_{\rm foc}-|\mathscr F_{\rm foc}| = \frac{M_{\rm foc}^2-|\mathscr F_{\rm foc}|^2}{M_{\rm foc}+|\mathscr F_{\rm foc}|} \geq \frac{M_{\rm foc}^2((\varsigma,\bm{y})^\top\mathsf C(\varsigma,\bm{y})+O((|\varsigma|+|\bm{y}|)^4))}{2M_{\rm foc}}.
\end{align*}
Since $\mathsf C$ is positive definite, there exist $r_0,c_0>0$ such that
\begin{equation*}
 M_{\rm foc}-|\mathscr F_{\rm foc}(\varsigma,\bm y)|\ge c_0M_{\rm foc}(|\varsigma|^2+|\bm y|^2)
\end{equation*}
whenever $|\varsigma|+|\bm y|\le r_0$. Since $(0,0)$ is a maximal point, the Hessian of $|\mathscr F_{\rm foc}|^2$ is negative definite throughout the neighborhood $|\varsigma|+|\bm y|\leq r_0$ given $r_0$ small enough.

For fixed $R>0$, we consider the possible maximizers inside the domain $\{(t,\bm{x}):|t-t_{{\rm focus},h}|\leq Rh, |\bm{x}-\bm{z}_\star|\leq Rh\}$. Choose $r_R>0$ smaller than both $r_0/2$ and $R/2$. By \eqref{eq:spacetime-strict-maximum} and compactness, outside the domain $\Omega_{r_R}:=\{(\varsigma,\bm y): |\varsigma|^2+|\bm y|^2 < r_R^2\}$, there exists $\delta_R>0$ such that
\begin{equation*}
 |\mathscr F_{\rm foc}(\varsigma,\bm y)|\le(1-\delta_R)M_{\rm foc}.
\end{equation*}
Equation \eqref{eq:focal-expansion} and \eqref{eq:focal-remainder-bound} show that 
$$
    E_{z,h}^{\rm back,sel}\bigl(t_{{\rm focus},h}+h\varsigma,\bm z_\star+h\bm y\bigr)-C_h\alpha_hc_h^{\rm av}\mathscr F_{\rm foc}(\varsigma,\bm{y}) = O(h),
$$
and
$$
    E_{z,h}^{\rm back,sel}\bigl(t_{{\rm focus},h},\bm z_\star\bigr)=C_h\alpha_hc_h^{\rm av}\mathscr F_{\rm foc}(0,0)+h\mathscr R_{h,R}(0,0).
$$
Using the uniform boundedness of remainder $\mathscr R_{h,R}$, we have that
$$
    E_{z,h}^{\rm back,sel}\bigl(t_{{\rm focus},h},\bm z_\star\bigr)\geq M_{\rm foc}-Ch.
$$
On the other hand, outside $\Omega_{r_R}$, 
\begin{align*}
|E_{z,h}^{\rm back,sel}\bigl(t_{{\rm focus},h}+h\varsigma,\bm z_\star+h\bm y\bigr)| = |C_h\alpha_hc_h^{\rm av}\mathscr F_{\rm foc}(\varsigma,\bm{y})+h\mathscr R_{h,R}| \leq (1-\delta_R)M_{\rm foc}+Ch.
\end{align*}
Letting $h<\delta_R M_{\rm foc}/(2C)$, we can ensure that when $(\varsigma,\bm{y})\notin\Omega_{r_R}$, 
$$
|E_{z,h}^{\rm back,sel}\bigl(t_{{\rm focus},h}+h\varsigma,\bm z_\star+h\bm y\bigr)|<|E_{z,h}^{\rm back,sel}\bigl(t_{{\rm focus},h},\bm z_\star\bigr)|,
$$
which means the maximizer must fall inside $\Omega_{r_R}$.

In $\Omega_{r_R}$, $|\mathscr F_{\rm foc}|\ge M_{\rm foc}/2$ after decreasing $r_0$ if necessary. Define
\begin{equation*}
 G_h(\varsigma,\bm y):=
 |\mathscr F_{\rm foc}(\varsigma,\bm y)+h\mathscr R_{h,R}(\varsigma,\bm y)|^2.
\end{equation*}
The Hessian of $G_0$ at the origin is $-2M_{\rm foc}^2\mathsf C$, which is negative definite. By the choice of $r_0$ and the $C^2$ estimate in \eqref{eq:focal-remainder-bound}, after decreasing $r_R$ if necessary, $-D^2G_h\ge cI_3$ throughout the ball of radius $r_R$ for all small $h$. Let $(\varsigma_h^\#,\bm y_h^\#)$ be the scaled coordinates of a maximizer. It is an interior point, so $\nabla G_h(\varsigma_h^\#,\bm y_h^\#)=0$. Taylor's formula gives
\begin{equation*}
 0=\nabla G_h(0,0)+\left(\int_0^1D^2G_h(s\varsigma_h^\#,s\bm y_h^\#)\dd s\right)(\varsigma_h^\#,\bm y_h^\#)^\top.
\end{equation*}
The first term is $O(h)$ since $\nabla G_0(0,0)=0$, and the matrix in parentheses has a uniformly bounded inverse. Therefore
\begin{equation*}
 |\varsigma_h^\#|+|\bm y_h^\#|\le C_Rh.
\end{equation*}
Since $t_h^\#=t_{{\rm focus},h}+h\varsigma_h^\#$ and $\bm x_h^\#=\bm z_\star+h\bm y_h^\#$, this proves \eqref{eq:focus-error}.

We now set $\varsigma=0$. Since $a_{\rm foc}$ is supported in a fixed compact set,
\begin{equation*}
 |\mathscr F_{\rm foc}(0,\bm y)-M_{\rm foc}|\le |\bm y|\int|\bm\xi|a_{\rm foc}(\bm\xi)\dd\bm\xi.
\end{equation*}
Thus there is $r_1>0$ such that
\begin{equation*}
 |\mathscr F_{\rm foc}(0,\bm y)|\ge\frac34M_{\rm foc}\qquad(|\bm y|\le r_1).
\end{equation*}
On the other hand, let $|\bm y|\ge1$ and choose $j\in\{1,2\}$ such that $|y_j|\ge|\bm y|/\sqrt2$. Since $a_{\rm foc}\in C_c^\infty$, integration by parts gives, for every $N$,
\begin{equation*}
 \mathscr F_{\rm foc}(0,\bm y)=(\ii y_j)^{-N}\int e^{\ii\bm y\cdot\bm\xi}\partial_{\xi_j}^N a_{\rm foc}(\bm\xi)\dd\bm\xi.
\end{equation*}
Hence there exists $C_N$ such that
\begin{equation*}
 |\mathscr F_{\rm foc}(0,\bm y)|\le C_N(1+|\bm y|)^{-N}.
\end{equation*}
Choose $R_1>r_1$ large enough, such that the right-hand side is at most $M_{\rm foc}/4$ for $|\bm y|\ge R_1$. Apply \eqref{eq:focal-expansion} on the fixed ball $B_{R_1}(0)$. For $|\bm y|\le r_1$, the scaled field has magnitude at least $3M_{\rm foc}/4-Ch$, while its value at the origin is at most $M_{\rm foc}+Ch$. Thus $\{|\bm y|\leq r_1\}\subset \mathcal F_{h,1/2}$. On $|\bm y|=R_1$, its magnitude is at most $M_{\rm foc}/4+Ch$, while the value at the origin is at least $M_{\rm foc}-Ch$, thus the boundary of $\mathcal F_{h,1/2}$ cannot cross $\partial B_{R_1}(0)$. Its area in the $\bm y$ variables is therefore bounded above and below by positive constants independent of $h$. The change of variables
$\bm x=\bm z_\star+h\bm y$ multiplies area by $h^2$, which proves \eqref{eq:focal-scale}.
\end{proof}

\section*{Acknowledgments} 
H. Liu has been partially supported by the Hong Kong RGC General Research Funds (projects 11311122, 12301420, and 11300821). W. Wu has been supported by the NSFC grant (12301539).

\textbf{Data availability.} No datasets were generated or analysed during the current study, because our work proceeds within a theoretical and mathematical approach.

\bibliographystyle{plain} \bibliography{time-modulated}

\begin{appendix}

\section{Semiclassical conventions}\label{sec:semi-convention}
Throughout the paper, we require $0<h\le h_0$ for a fixed constant $h_0\in\mathbb{R}_+$. We use
\begin{equation*}
    \mathcal H_{\rm s}:=L^2(\mathbb R^2;\mathbb C),
    \qquad
    \mathcal H_{\rm v}:=L^2(\mathbb R^2;\mathbb C^4).
\end{equation*}

A scalar or matrix family $a_h(\bm x,\bm\xi)$, or a vector family $\bm a_h(\bm x,\bm\xi)$, is called a symbol of order $\mathfrak s$ on an open set $\mathcal V\subset T^*\mathbb R^2$ if, for every compact $K\Subset\mathcal V$ and all multi-indices $\alpha,\beta$. The estimates below are written for $a_h$, and vector-valued estimates are understood componentwise.
\begin{equation}\label{eq:symbol-class-definition}
 \sup_{(\bm x,\bm\xi)\in K}
 \langle\bm\xi\rangle^{-\mathfrak s+|\beta|}
 \bigl|\partial_{\bm x}^{\alpha}\partial_{\bm\xi}^{\beta}
 a_h(\bm x,\bm\xi)\bigr|
 \le C_{K,\alpha,\beta},
 \qquad 0<h\le h_0.
\end{equation}
For vector and matrix symbols, $|\cdot|$ is any fixed Euclidean matrix norm. If all $a_h$ are supported in one fixed compact subset of $\mathcal V$, we call the family compactly supported. For such a family, we use the simpler finite bound
\begin{equation}\label{eq:symbol-seminorm-definition}
 \|a_h\|_{\mathfrak s,N}:=
 \max_{|\alpha|+|\beta|\le N}
 \sup_{(\bm x,\bm\xi)}
 \langle\bm\xi\rangle^{-\mathfrak s+|\beta|}
 \bigl|\partial_{\bm x}^{\alpha}\partial_{\bm\xi}^{\beta}
 a_h(\bm x,\bm\xi)\bigr|.
\end{equation}
If symbol family $a_h$ has the bounds in \eqref{eq:symbol-class-definition}, or equivalently the quantities in \eqref{eq:symbol-seminorm-definition} after a compact cutoff, uniform in $h$, we call $a_h$ \emph{uniformly bounded as an order-$\mathfrak s$ symbol}. 

We apply Kohn--Nirenberg quantization in this paper. For symbol $a$, the \emph{(left) quantization} is defined as
$$
    (\Op_h(a)u)(\bm{x}) = \frac{1}{(2\pi h)^2}\int_{\mathbb{R}^2}\int_{\mathbb{R}^2} e^{\ii(\bm{x}-\bm{y})\cdot\bm{\xi}/h}a(\bm{x},\bm{\xi})u(\bm{y})\dd\bm{y}\dd\bm{\xi}.
$$

We write $H_h^s(\mathbb R^2;\mathbb C^d)$ for the semiclassical Sobolev space with norm $\|\langle hD\rangle^su\|_{L^2}$. If $a_h$ is a matrix symbol of order $\mathfrak s$, then
\begin{equation*}
    \Op_h(a_h):H_h^s(\mathbb R^2;\mathbb C^{d_1})
    \longrightarrow H_h^{s-\mathfrak s}(\mathbb R^2;\mathbb C^{d_2})
\end{equation*}
is uniformly bounded. In this paper, an order-$\mathfrak s$ semiclassical pseudodifferential operator means a properly supported quantization of an order-$\mathfrak s$ symbol, up to a negligible smoothing operator.

More precisely, a remainder $R_h$ is called \emph{negligible} from $H_h^{-N}(\mathbb R^2;\mathbb C^{d_1})$ to $H_h^N(\mathbb R^2;\mathbb C^{d_2})$ if, for every $M$,
\begin{equation*}
 \|R_h\|_{H_h^{-N}(\mathbb R^2;\mathbb C^{d_1})
 \to H_h^N(\mathbb R^2;\mathbb C^{d_2})}
 \le C_{N,M}h^M.
\end{equation*}
Since $H_h^N(\mathbb{R}^2;\mathbb{C}^d)$ will be repeatedly used in the main text when introducing negligible remainders, we abbreviate $H_h^N(\mathbb{R}^2)$ as $H_h^N$, and $H_h^N(\mathbb R^2; \mathbb C^d)$ as $H_h^{N,d}$. 

Symbols initially defined only on $\mathcal U$ are multiplied by a fixed cutoff in a slightly larger compact subset of $\mathcal U$ and then extended globally. Different extensions give the same microlocal operator modulo a negligible remainder. By saying \emph{microlocally on $K\Subset\mathcal U$}, we insert fixed compactly supported semiclassical cutoffs on the input near $K$ and on the output near the corresponding flow-out.

We shall also use amplitudes $a_h(\bm x,\bm\xi)$ whose frequency support is contained in one fixed compact set but whose spatial bounds are only local. We call such a family a \emph{controlled amplitude} if, for every compact $K_{\bm x}\subset\mathbb R^2$ and every $N$,
\begin{equation*}
 \max_{|\alpha|+|\beta|\le N}\sup_{\bm x\in K_{\bm x},\bm\xi}\bigl|\partial_{\bm x}^{\alpha}\partial_{\bm\xi}^{\beta}a_h(\bm x,\bm\xi)\bigr|\le C_{K_{\bm x},N}.
\end{equation*}
If an additional parameter $\bm y$ ranges in a fixed compact set, we call $a_h$ \emph{uniformly controlled in $\bm y$} if the same estimate holds after any fixed number of $\bm y$-derivatives.

Finally, by \emph{an order-zero propagating operator along $\bm\Phi^t$}, we mean a semiclassical operator $A_h(t)$ which, after a finite partition in phase space, can be written in the form 
$$
    (A_h(t)u)(\bm{x}) = \frac{1}{(2\pi h)^d}\int_{\mathbb{R}^d}\int_{\mathbb{R}^d} e^{\ii(S(t,\bm{x},\bm{\eta})-\bm{y}\cdot\bm{\eta})/h}a(t,\bm{x},\bm{\eta};h)u(\bm{y})\dd\bm{y}\dd\bm{\eta}.
$$
Here $S(t,\bm{x},\bm{\eta})$ generates $\bm{\Phi}^t$, i.e. 
$$
\bm{y}=\partial_{\bm\eta}S(t,\bm{x},\bm{\eta}), \qquad \bm{\xi}=\partial_{\bm{x}}S(t,\bm{x},\bm{\eta})
$$ 
if and only if 
$$
(\bm{x},\bm{\xi})=\bm{\Phi}^t(\bm{y},\bm{\eta}).
$$ 
The amplitude $a(t,\bm{x},\bm{\eta};h)$ is supported in a fixed compact set, independent of $h$ and $t$, and for every multi-indices $\alpha, \beta$ and every integer $j>0$,
$$
    |\partial_t^j\partial_{\bm{x}}^\alpha\partial_{\bm{\eta}}^\beta a(t,\bm{x},\bm{\eta};h)|\leq C_{j\alpha\beta}, 
$$
where the constants $C_{j\alpha\beta}$ are independent of $h$ and $t$. If $a$ is matrix-valued, the absolute value is replaced by any fixed matrix norm.

\end{appendix}

\end{document}